\documentclass[amstex,12pt]{amsart}
\usepackage{amssymb}
\usepackage{amsmath}
\usepackage{amscd}
\usepackage{mathrsfs}
\usepackage{graphicx}
\usepackage{latexsym}

\usepackage[all]{xy}

\theoremstyle{definition} 

\newtheorem*{definition}{Definition} 
\newtheorem*{remark}{Remark}

\theoremstyle{plain}  

\newtheorem{theorem}{Theorem}[section]
\newtheorem{proposition}[theorem]{Proposition}
\newtheorem{corollary}[theorem]{Corollary}
\newtheorem{lemma}[theorem]{Lemma}

\DeclareMathOperator{\id}{id}

\DeclareMathOperator{\esssup}{\mbox{\rm{ess.sup}}}
\DeclareMathOperator{\QC}{QC}
\DeclareMathOperator{\AC}{AC}
\DeclareMathOperator{\QS}{QS}

\DeclareMathOperator{\Sym}{Sym}
\DeclareMathOperator{\Mob}{\mbox{\rm{M\"ob}}}
\DeclareMathOperator{\Diff}{Diff}

\DeclareMathOperator{\Homeo}{Homeo}
\DeclareMathOperator{\Hol}{Hol_2}
\DeclareMathOperator{\Bel}{Bel}

\DeclareMathOperator{\modulus}{mod}
\DeclareMathOperator{\D}{\mathbb D}
\DeclareMathOperator{\C}{\mathbb C}
\DeclareMathOperator{\R}{\mathbb R}

\DeclareMathOperator{\H2}{\mathbb H}
\DeclareMathOperator{\S1}{\mathbb S}
\DeclareMathOperator{\Chat}{\widehat {\mathbb C}}

\begin{document}

\title[Teich\-m\"ul\-ler space and diffeomorphisms]
{Teich\-m\"ul\-ler space of circle diffeomorphisms with H\"older continuous derivative}

\author[K. Matsuzaki]{Katsuhiko Matsuzaki}
\address{Department of Mathematics, School of Education, Waseda University, \endgraf
Shinjuku, Tokyo 169-8050, Japan}
\email{matsuzak@waseda.jp}

\makeatletter
\@namedef{subjclassname@2010}{%
\textup{2010} Mathematics Subject Classification}
\makeatother

\thanks{This work was supported by JSPS KAKENHI 25287021.}

\begin{abstract}
Based on the quasiconformal theory of the universal Teich\-m\"ul\-ler space,
we introduce the Teich\-m\"ul\-ler space of
diffeomorphisms of the unit circle with $\alpha$-H\"older continuous derivatives
as a subspace of the universal Teich\-m\"ul\-ler space. 
We characterize such a diffeomorphism quantitatively
in terms of the complex dilatation of its quasiconformal extension and
the Schwarzian derivative given by the Bers embedding.
Then, we provide a complex Banach manifold structure for it and
prove that its topology coincides with the one induced by local $C^{1+\alpha}$-topology at the base point.
\end{abstract}
 
\maketitle

\section{Introduction}\label{1}

Parametrization of orientation-preserving diffeomorphisms of the unit circle $\S1$ can be studied 
in the framework of the theory of Teich\-m\"ul\-ler spaces. In this case,
we utilize the universal Teich\-m\"ul\-ler space $T$, which can be regarded as the group $\QS$ of
quasisymmetric self-homeomorphisms 
of $\S1$ modulo post-composition of M\"obius transformations $\Mob(\S1)$. 
Here, a quasisymmetric self-homeomorphism of $\S1$ is the boundary extension of a quasiconformal self-homeomorphism of 
the unit disk $\D$.
The Teich\-m\"ul\-ler space of an arbitrary hyperbolic Riemann surface 
can be understood as the fixed point locus of the corresponding
Fuchsian group acting on $T$. 
If we replace the group invariance with certain regularity conditions for quasisymmetric homeomorphisms,
we can also embed the Teich\-m\"ul\-ler space of such a family of circle homeomorphisms in the 
universal Teich\-m\"ul\-ler space $T=\Mob(\S1) \backslash \QS$.

In this paper, we formulate the Teich\-m\"ul\-ler space $T^\alpha_0=\Mob(\S1) \backslash \Diff_+^{1+\alpha}(\S1)$ of circle diffeomorphisms 
with H\"older continuous derivatives of exponent $\alpha \in (0,1)$. 
We provide a complex Banach manifold structure for $T^\alpha_0$ and 
prove basic properties of this space.
The arguments for $T^\alpha_0$ are modeled on those for the universal Teich\-m\"ul\-ler space $T$
and certain refinements are imported from the theory for the little Teich\-m\"ul\-ler subspace $T_0=\Mob(\S1) \backslash \Sym$.
Here, the subgroup $\Sym \subset \QS$ consists of symmetric self-homeomorphisms $\S1$, which are 
the boundary extension of asymptotically conformal homeomorphisms of $\D$ whose complex dilatations vanish at the boundary $\S1$.
It contains all circle diffeomorphisms; hence, $\Diff_+^{1+\alpha}(\S1) \subset \Sym$.
We will survey necessary results on the universal Teich\-m\"ul\-ler space $T$ in Section \ref{2}
and on the little subspace $T_0$ in Section \ref{3}.

We first characterize circle diffeomorphisms with H\"older continuous derivatives
in terms of their quasiconformal extension to $\D$. 
This originates from the work of asymptotically conformal maps by Carleson \cite{Car}.
Later, Gardiner and Sullivan \cite{GS} developed
the theory of symmetric homeomorphisms of $\S1$
using previous results on quasiconformal extension and Schwarzian derivatives of univalent functions
in Becker and Pommerenke \cite{BP}. We will refine these results quantitatively concerning the
decay order of the corresponding maps vanishing at the boundary. 

We verify in Section \ref{4} that,
if the complex dilatation $\mu$ of an asymptotically conformal homeomorphism of $\D$ decays
in the order of $O((1-|z|)^\alpha)$ as $|z| \to 1$, then the hyperbolically weighted Schwarzian derivative 
of the developing map of the projective structure on the exterior disk $\D^*$ determined by $\mu$
decays exactly in the same order $\alpha$. 
This is carried out by
dividing the support of $\mu$ suitably into annular regions and estimating the pre-Schwarzian derivative of
the composition of conformal homeomorphisms.
A different proof was previously obtained by Dyn$'$kin \cite{Dy}, but we have to prepare
more precise estimates in terms of a weighted supremum norm of $\mu$ (Theorem \ref{basic1} and Corollary \ref{interior}).

In Section \ref{5}, we mainly consider one-dimensional properties of 
circle diffeomorphisms 
with H\"older continuous derivatives.
We first provide a topology for $\Diff_+^{1+\alpha}(\S1)$ 
and see that it is a topological group (Proposition \ref{top}).
The topology is defined in a neighborhood of the identity map by $C^{1+\alpha}$-convergence and then
distributed to every point by the right translation of the group.
For the characterization of an element of $\Diff_+^{1+\alpha}(\S1)$, 
a result of Carleson \cite{Car} plays an important role, as it
gives a connection between the H\"older continuity of the derivative and the quasisymmetry quotient of $g$.
We review his theorem and supply necessary claims for our arguments.

A fundamental result is that if the complex dilatation of
an asymptotically conformal homeomorphism of $\D$ decays 
in the order of $O((1-|z|)^\alpha)$, then the regularity of its boundary extension $g$ to $\S1$ is exactly $C^{1+\alpha}$.
Carleson found that it is at least $C^{1+\alpha/2}$.
This problem was investigated further by Anderson and Hinkkanen \cite{AH} among others,
and settled qualitatively by Dyn$'$kin \cite{Dy} 
and Anderson, Cant\'on, and Fern\'andez \cite{ACF}.
Combined with the aforementioned results, this can be summarized 
as follows (Theorem \ref{main}). 

\begin{theorem}\label{th1}
Let $\alpha$ be a constant with $0<\alpha<1$.
The following conditions are equivalent for $g \in \QS$:
\begin{itemize}
\item[$(1)$]
$g$ is a diffeomorphism of $\S1$ with H\"older continuous derivative of exponent $\alpha$;
\item[$(2)$]
$g$ extends continuously to a quasiconformal self-homeomorphism of $\D$ whose complex dilatation $\mu(z)$ decays
in the order of $O((1-|z|)^\alpha)$ as $z \in \D$ tends to the boundary;
\item[$(3)$]
the Schwarzian derivative $\varphi(z)$ of the conformal homeomorphism of $\D^*$ determined by $g$ behaves
in the order of $O((|z|-1)^{-2+\alpha})$ as $z \in \D^*$ tends to the boundary.
\end{itemize}
\end{theorem}

In Section \ref{6}, 
we will improve on Theorem \ref{th1} 
with a different proof, which is necessary for the arguments of Teichm\"uller spaces
(Theorem \ref{key}).
Our strategy is to represent a circle diffeomorphism $g$ by conformal welding, which was originally 
proposed by Anderson, Becker, and Lesley \cite{ABL}. For the argument in this method, we need to know that
an asymptotically conformal self-homeomorphism $f$ of $\D$ and its inverse mapping $f^{-1}$ have the complex dilatations
of order $O((1-|z|)^\alpha)$ at the same time.
For this purpose,
we extend the consequence of the Mori theorem to
a quasiconformal self-homeomorphism $f$ of $\D$ with complex dilatation 
of order $O((1-|z|)^\alpha)$. The result is that $1-|f(z)|$ is comparable to $1-|z|$
without the power of the maximal dilatation $K(f)$ (Theorem \ref{mori}). This guarantees that the complex dilatation of
$f^{-1}$ is also of order $O((1-|z|)^\alpha)$.

\begin{theorem}\label{th0}
Let $f$ be a quasiconformal self-homeomorphism of $\D$ with $f(0)=0$ 
whose complex dilatation $\mu(z)$ satisfies $|\mu(z)| \leq \ell (1-|z|)^\alpha$
almost every $z \in \D$ for some $\ell \geq 0$.
Then, there is a constant $A \geq 1$ depending only on $K(f)$, $\alpha$, and $\ell$ such that
$$
\frac{1}{A}(1-|z|) \leq 1-|f(z)| \leq A(1-|z|)
$$
for every $z \in \D$.
\end{theorem}

For Beltrami coefficients and Schwarzian derivatives as above, we prepare the following spaces:
$\Bel^\alpha_0(\D)$ is the space of Beltrami coefficients $\mu$ on $\D$ with a finite norm 
$\Vert \mu \Vert_{\infty,\alpha}=\esssup \rho_{\D}^\alpha(z)|\mu(z)|$; and $B^\alpha_0(\D^*)$ is 
the Banach space of holomorphic quadratic differentials $\varphi=\varphi(z)dz^2$ on $\D^*$ with a finite norm
$\Vert \varphi \Vert_{\infty,\alpha}=\sup \rho_{\D^*}^{-2+\alpha}(z)  |\varphi(z)|$.
Here, $\rho_{\bullet}$ denotes the hyperbolic density of each space.
Then, Theorem \ref{th1} implies that the Teich\-m\"ul\-ler projection $\pi$,
the Bers projection $\Phi$, and the Bers embedding $\beta$ for the universal Teich\-m\"ul\-ler space $T$
also work for our spaces by restriction of the original mappings:
$$
\xymatrix{
& \Bel_0^\alpha(\D) \ar[ld]^{\pi} \ar[rd]_{\Phi} & \\
T_0^\alpha=\Mob(\S1) \backslash  \Diff_+^{1+\alpha}(\S1)  \ar[rr]_{\beta} & & \beta(T) \cap B_0^\alpha(\D^*) \\
}
$$

The topology on $T_0^\alpha$ is induced from $\Bel_0^\alpha(\D)$ by $\pi$.
If we regard $T_0^\alpha$ as the subgroup of $\Diff_+^{1+\alpha}(\S1)$ consisting of normalized elements, then we can also provide it with
the right uniform topology of $\Diff_+^{1+\alpha}(\S1)$,
which is generated by the right translations
of the local $C^{1+\alpha}$-topology
at the identity. 
In Section \ref{8}, we prove that these topologies on $T_0^\alpha$ are the same
(Theorem \ref{topology}).

\begin{theorem}\label{th2}
The quotient topology on $T_0^\alpha$
induced by $\pi:\Bel_0^\alpha(\D) \to T_0^\alpha$ coincides with the right uniform topology of $\Diff_+^{1+\alpha}(\S1)$ and, 
in particular, 
$T_0^\alpha$ is a topological group.
\end{theorem} 

The complex structure on $T_0^\alpha$ is given by showing that
the Bers embedding $\beta$ as above is a homeomorphism onto its image.
Moreover, we want to find that the base point change map (right translation) of $T_0^\alpha$ is 
compatible with this complex structure. To this end, we will prove that
the Bers projection $\Phi$ is a holomorphic split submersion.
To see that $\Phi$ is continuous, we use an integral representation of the Schwarzian derivative $\Phi(\mu)$,
which was originally proposed by Astala and Zinsmeister \cite{AZ}. Then, a careful estimate of this integral taking the dependence of constants
into account yields the assertion on continuity. The holomorphy is a consequence from the continuity in our situation.
To see that $\Phi$ is a split submersion, we construct a local holomorphic section of $\Phi$.
For the universal Teich\-m\"ul\-ler space (and Teich\-m\"ul\-ler spaces of Riemann surfaces), this was originally proved by Bers \cite{Ber},
and afterward certain modifications have been made to develop a standard argument.
We adapt this argument to our situation to show the continuity 
with respect to the topology in our spaces. In Section \ref{8}, we will prove the following:

\begin{theorem}\label{th3}
The Bers projection $\Phi:\Bel^\alpha_0(\D) \to B_0^\alpha(\D^*)$ is
a holomorphic split submersion onto its image. This implies that
the Bers embedding $\beta:T_0^\alpha \to B_0^\alpha(\D^*)$ is
a homeo\-morphism onto its image. With this complex structure of $T_0^\alpha$ identified with
a domain of the complex Banach space $B_0^\alpha(\D^*)$, 
every base point change map of $T_0^\alpha$ is a biholomorphic automorphism of $T_0^\alpha$.
\end{theorem}

A motivation of this work is to apply the Bers embedding of the Teich\-m\"ul\-ler space $T_0^\alpha$ to
studies of the rigidity of the $\Diff_+^{1+\alpha}(\S1)$-representation of a M\"obius group and the regularity of the conjugation 
of a subgroup of $\Diff_+^{1+\alpha}(\S1)$ to a M\"obius group. These arguments are developed 
in a continuation \cite{Mat} of the present work. An overview of our project can be found in \cite{Mat1}.
A preliminary study can be found in \cite{Mat0}.


\section{The universal Teich\-m\"ul\-ler space}\label{2}
In this section, we define the universal Teich\-m\"ul\-ler space in terms of the group of 
quasisymmetric self-homeomorphisms of the circle, and then introduce a topological and a complex structure on
this space by using the quasiconformal theory: the Beltrami equation and the Schwarzian derivative.
Basic results can be found in Lehto \cite{Leh}.

We denote the group of all quasiconformal self-homeomorphisms of the unit disk $\D$ by
$\QC(\D)$. Each quasiconformal homeomorphism $f \in \QC(\D)$ extends continuously to 
the boundary $\S1$ as a homeomorphism. Then, we have a homomorphism
$q:\QC(\D) \to \Homeo(\S1)$ in the group of self-homeomorphisms of the unit circle $\S1$.
An orientation-preserving self-homeomorphism $g$ of $\S1$ is called {\it quasi\-symmetric} if
$g \in {\rm Im}\,q$. We denote the group ${\rm Im}\, q$ of all quasisymmetric self-homeomorphisms of $\S1$ by $\QS$. 
Let $\Mob(\D) \subset \QC(\D)$ denote the subgroup 
of all conformal self-homeomorphisms of $\D$, which are M\"obius transformations of $\D$.
We define  $\Mob(\S1)=q(\Mob(\D)) \subset \QS$.

\begin{definition}
The {\it universal Teich\-m\"ul\-ler space} $T$ is defined as 
the set of the cosets $\Mob(\S1) \backslash \QS$. We denote the coset of $g \in \QS$ by $[g]$.
\end{definition}

The {\it Beltrami coefficient} $\mu$ on a domain $D \subset \Chat$ is a measurable function with a 
supremum norm $\Vert \mu \Vert_\infty$ less than $1$.
We denote the set of all Beltrami coefficients on $D$ by
$$
\Bel(D)=\{\mu \in L^\infty(D) \mid \Vert \mu \Vert_\infty<1 \}.
$$
Every quasiconformal homeomorphism
$f:D \to D'$ has partial derivatives $\partial f$ and $\bar \partial f$ in the distribution sense and
the ratio $\mu_f(z)=\bar \partial f(z)/\partial f(z)$ called the {\it complex dilatation} belongs to $\Bel(D)$.
The {\it maximal dilatation} of $f$ is defined by
$$
K(f)=\frac{1+\Vert \mu_f \Vert_\infty}{1-\Vert \mu_f \Vert_\infty}.
$$
Given $K \geq 1$, we call $f$ a $K$-quasiconformal if $K(f) \leq K$.
The {\it measurable Riemann mapping theorem} asserts that
a Beltrami coefficient uniquely determines a quasiconformal homeomorphism up to post-composition of
conformal homeomorphisms (see Lehto and Virtanen \cite{LV} for the history of this theorem, and
Morrey \cite{Mor},
Ahlfors and Bers \cite{AB}, and Ahlfors \cite{Ah0} for the proof).

Applying this theorem to quasiconformal homeomorphisms of $\D$, we see that
$\Bel(\D)$ can be identified with the set of the cosets $\Mob(\D) \backslash \QC(\D)$.
Then, the boundary extension $q:\QC(\D) \to \QS$ induces a surjective map
$\pi:\Bel(\D) \to T$ by taking the quotient of $\Mob(\D) \cong \Mob(\S1)$.
This is called the {\it Teich\-m\"ul\-ler projection}.
The topology of the universal Teich\-m\"ul\-ler space $T$ is given as the quotient topology
of the unit ball $\Bel(\D)$ of the Banach space $L^\infty(\D)$ by the projection $\pi$ so that
$\pi$ is continuous. 

There is a global continuous section for the Teich\-m\"ul\-ler projection $\pi:\Bel(\D) \to T$.
This is defined by giving a canonical quasiconformal extension $e:\QS \to \QC(\D)$
for each quasisymmetric self-homeomorphism $g$ of $\S1$. The extension due to
Beurling and Ahlfors \cite{BA} can be used to obtain such a section. Douady and Earle \cite{DE} introduced
another extension $e_{\rm DE}:\QS \to \QC(\D)$
having the conformal naturality such that
$$
e_{\rm DE}(\phi_1 \circ g \circ \phi_2)=e_{\rm DE}(\phi_1) \circ e_{\rm DE}(g) \circ e_{\rm DE}(\phi_2)
$$
for any $\phi_1, \phi_2 \in \Mob(\S1)$ and any $g \in \QS$.
We note that $e_{\rm DE}(\phi_1)$ and $e_{\rm DE}(\phi_2)$ are
the M\"obius transformations of $\D$ extending $\phi_1$ and $\phi_2$, respectively.
By taking the quotient of $\Mob(\S1) \cong \Mob(\D)$, 
we have a continuous map $s_{\rm DE}:T \to \Bel(\D)$ such that
$\pi \circ s_{\rm DE}=\id_T$. We call this the {\it conformally natural section}.
The existence of a global continuous section implies that $T$ is contractible.

The measurable Riemann mapping theorem implies that, for every $\nu \in \Bel(\D)$, there is a unique normalized
quasiconformal homeomorphism $f \in \QC(\D)$ whose complex dilatation coincides with $\nu$.
Here, the {\it normalization} is given by fixing three boundary points $1$, $i$, and $-1$ on $\S1$.
We denote this normalized quasiconformal homeomorphism by $f^\nu$. 
The subgroup of $\QC(\D)$ consisting of all normalized elements is defined as $\QC_*(\D)$.
This also defines the normalized elements of $\QS$, which constitute the subgroup $\QS_*=q(\QC_*(\D))$.

Applying this normalization,
we can define a group structure on $\Bel(\D)$ and $T$ as follows.
For any $\nu_1, \nu_2 \in \Bel(\D)$, we set
$\nu_1 \ast \nu_2$ to be the complex dilatation of the composition $f^{\nu_1} \circ f^{\nu_2}$.
Then, $\Bel(\D)$ has a group structure with this operation $\ast$. In other words,
by the identification of $\Bel(\D)$ with $\QC_*(\D)$, we regard $\Bel(\D)$ as a subgroup of $\QC(\D)$.
We denote the inverse element of $\nu \in \Bel(\D)$ by $\nu^{-1}$, which is the complex dilatation of $(f^\nu)^{-1}$.
The chain rule of partial differentials yields a formula
$$
\nu_1 \ast \nu_2^{-1}(\zeta)=\frac{\nu_1(z)-\nu_2(z)}{1-\overline{\nu_2(z)}\nu_1(z)}\cdot\frac{\partial f^{\nu_2}(z)}{\overline{\partial f^{\nu_2}(z)}}
\qquad(\zeta=f^{\nu_2}(z)).
$$

For the base point $[\id]$ of $T$, the inverse image of the Teich\-m\"ul\-ler projection
$$
\pi^{-1}([\id])=\{\nu \in \Bel(\D) \mid q(f^{\nu})=\id\}
$$
is a normal subgroup of $\Bel(\D)$ as $q:\QC(\D) \to \QS$ is a homomorphism.
Having $T=\Bel(\D)/\pi^{-1}([\id])$, we see that $T$ has a group structure with the operation $\ast$ 
defined by $\pi(\nu_1) \ast \pi(\nu_2)=\pi(\nu_1 \ast \nu_2)$. Then, 
$\pi:\Bel(\D) \to T$ is a surjective homomorphism with $\pi^{-1}([\id])$ its kernel.
If we identify $T$ with $\QS_*$, we may regard $T$ as a subgroup of $\QS$ and the projection $\pi$ as the restriction of $q$ to
$\QC_*(\D)$.

Each $\nu\in \Bel(\D)$ induces the right translation $r_\nu:\Bel(\D) \to \Bel(\D)$ by
$\mu \mapsto \mu \ast \nu^{-1}$. The projection under $\pi$ yields a well-defined map
$R_{\pi(\nu)}:T \to T$ by 
$$
\pi(\mu) \mapsto \pi(\mu \ast \nu^{-1})=\pi(\mu)\ast \pi(\nu)^{-1}.
$$
In this way, for every point $\tau \in T$,
we have the base point change map $R_\tau:T \to T$ sending $\tau$ to
$[\id]$.
By the above formula, we see that $r_\nu$ and $(r_\nu)^{-1}=r_{\nu^{-1}}$ are continuous; 
hence, $r_\nu$ is a homeomorphism 
onto $\Bel(\D)$. 
From this, we see that the base point change map $R_\tau$ is also a homeomorphism onto $T$.

The universal Teich\-m\"ul\-ler space $T$ has a complex structure modeled on
a certain complex Banach space. This is seen as follows.
For $\mu \in \Bel(\D)$, we
extend $\mu(z)$ to $\Chat$
by setting $\mu(z) \equiv 0$ for $z \in \mathbb D^*=\Chat-\overline{\D}$.
By the measurable Riemann mapping theorem, there exists a unique quasiconformal self-homeomorphism 
$f_\mu$ of $\Chat$ up to post-composition of M\"obius transformations whose complex dilatation coincides with
the extended Beltrami coefficient $\mu$. 
We take the Schwarzian derivative $S_f(z)$
of the conformal homeomorphism $f(z)=f_\mu|_{\D^*}(z)$ on $\D^*$.
The ambiguity of $f_\mu$ by M\"obius transformations is offset by taking the
Schwarzian derivative because 
$S_{h \circ f}(z)= S_f(z)$ for every $h \in \Mob(\Chat)$.

We define the Banach space of holomorphic quadratic differentials $\varphi=\varphi(z)dz^2$
on $\D^*$ with a finite hyperbolic supremum norm by
$$
B(\D^*)=\{\varphi \in \Hol(\D^*) \mid \Vert \varphi \Vert_\infty =\sup_{z \in \D^*} \rho^{-2}_{\D^*}(z)|\varphi(z)|<\infty\},
$$
where $\rho_{\D^*}(z)=2/(|z|^2-1)$ is the hyperbolic density on $\D^*$.
We note that an element $\varphi$ of $\Hol(\D^*)$ satisfies $\varphi(z)=O(1/z^4)$ $(z \to \infty)$.
The Nehari--Kraus theorem asserts that $\Vert \varphi \Vert_\infty  \leq 3/2$ for
the Schwarzian derivative $\varphi(z)=S_f(z)$ of any conformal homeomorphism $f$ of $\D^*$.
Hence, we have a map
$\Phi:\Bel(\D) \to B(\D^*)$ by the correspondence of $\mu \in \Bel(\D)$ to $S_{f_\mu|_{\D^*}}$,
which is called the {\it Bers projection} (onto the image $\Phi(\Bel(\D))$). 

With regard to the Teich\-m\"ul\-ler projection $\pi:\Bel(\D) \to T$ and
the Bers projection $\Phi:\Bel(\D) \to B(\D^*)$, it can be proved that
$\pi(\mu_1)=\pi(\mu_2)$ if and only if $\Phi(\mu_1)=\Phi(\mu_2)$.
Therefore, we have a well-defined injection $\beta:T \to B(\D^*)$ that satisfies $\beta \circ \pi=\Phi$.
This is called the {\it Bers embedding} of the universal Teich\-m\"ul\-ler space $T$.

\begin{proposition}\label{B-conti}
The Bers projection $\Phi:\Bel(\D) \to B(\D^*)$ is continuous. 
\end{proposition}

\begin{proof}
For two arbitrary points $\mu, \nu \in \Bel(\D)$, we apply the right translation $r_\nu$ to $\mu$.
On the quasidisk $f_{\nu}(\D^*)$, we use an estimate of the Schwarzian derivative of the
conformal homeomorphism $f_\mu \circ f_{\nu}^{-1}$
in terms of $\Vert r_\nu(\mu) \Vert_\infty$ (see Theorem II.3.2 in \cite{Leh}).
Then, 
$$
\Vert \Phi(\mu)-\Phi(\nu) \Vert_\infty \leq 3\Vert r_\nu(\mu)\Vert_\infty \leq 
\frac{3\Vert \mu-\nu\Vert_\infty}{1-\Vert \nu \Vert_\infty \Vert \mu \Vert_\infty},
$$
which implies that $\Phi$ is continuous.
\end{proof}

In fact, the Bers projection $\Phi$ is holomorphic. Once we have $\Phi$ as continuous, then
the holomorphy is a consequence of the point-wise holomorphic dependance of 
the normalized solution $f_\mu(z)$ of the Beltrami equation for $\mu$,
which is a significant contribution to the measurable Riemann mapping theorem
by Ahlfors and Bers \cite{AB}. 
Moreover, the following result was proved by Bers \cite{Ber}.

\begin{theorem}\label{hol-sub}
The Bers projection $\Phi:\Bel(\D) \to B(\D^*)$ is a holomorphic split submersion.
\end{theorem}

The condition needed for $\Phi$ to be a holomorphic split submersion is equivalent to the existence of
a local holomorphic section for $\Phi$ at every $\varphi \in \Phi(\Bel(\D))$ sending $\varphi$ to an arbitrary point of
$\Phi^{-1}(\varphi)$ (see Section 1.6 of Nag \cite{Nag} concerning holomorphic split submersion between domains of Banach spaces).
This implies that $\Phi$ is an open map, and in particular, the image $\Phi(\Bel(\D))$ in $B(\D^*)$ is open
(hence, it is a bounded domain).

As $\pi$ is a topological quotient map and $\Phi$ is continuous and open, the Bers embedding 
$\beta=\Phi \circ \pi^{-1}:T \to B(\D^*)$ is a homeomorphism onto the image 
$\beta(T)=\Phi(\Bel(\D))$. By identifying $T$ with a bounded domain $\beta(T) \subset B(\D^*)$,
we provide a complex structure for $T$. Then, the base point change map $R_\tau$ for every $\tau \in T$ is
a biholomorphic automorphism of $T$. Indeed, 
for an arbitrary point $\varphi \in \beta(T)$,
we take a local holomorphic section $\eta$ of $\Phi$ 
and $\nu \in \Bel(\D)$ with $\pi(\nu)=\tau$.
We represent $R_\tau$ at $\beta^{-1}(\varphi)$ by
$$
R_\tau=\beta^{-1} \circ \Phi \circ r_\nu \circ \eta \circ \beta.
$$
As $\Phi \circ r_\nu \circ \eta$ is holomorphic, 
$R_\tau$ is holomorphic. As the inverse $R_{\tau}^{-1}=R_{\tau^{-1}}$
is also holomorphic, $R_\tau$ is biholomorphic.

\section{Symmetric homeomorphisms and the little Teich\-m\"ul\-ler subspace}\label{3}

A quasisymmetric homeomorphism was originally introduced as a function on $\R$ that has quasiconformal extension to
the upper half-plane $\H2$. 
It can be characterized by the quasisymmetry quotient defined as follows:

\begin{definition}
An increasing homeomorphism $h: \R \to \R$ is called a 
{\it quasisymmetric function} if there exists a constant $M \geq 1$
such that 
$$
\frac{1}{M} \leq \frac{h(x+t)-h(x)}{h(x)-h(x-t)} \leq M
$$
holds for every $x \in \R$ and for every $t>0$. The ratio in the mid-term is called 
the {\it quasisymmetry quotient} of $h$ and is denoted by $m_h(x,t)$.
\end{definition}

For an orientation-preserving self-homeomorphism $g:\S1 \to \S1$,
we can take its lift $\widetilde g: \R \to \R$ with $u \circ \widetilde g=g \circ u$ 
for the universal cover $u: \R \to \S1$ given by $u(x)=e^{2\pi ix}$.
This is uniquely determined up to an additive integer and
is an increasing homeomorphism of $\R$ satisfying $\widetilde g(x+1)=\widetilde g(x)+1$.
Conversely, for an increasing homeomorphism $h: \R \to \R$ with $h(x+1)=h(x)+1$,
we can take its projection $\underline h:\S1 \to \S1$ with $u \circ h=\underline h \circ u$.

It is known that $g$ is a quasisymmetric self-homeomorphism of $\S1$ if and only if its lift
$\widetilde g$ is a quasisymmetric function on $\R$ (see Theorem 4.4 in \cite{Mat0}). 
To determine whether $\widetilde g$ is quasisymmetric, it is enough to check the quasisymmetry quotient
$m_{\widetilde g}(x,t)$ for $0 \leq x < 1$ and $0<t \leq 1/2$ (see Proposition 4.5 in \cite{Mat0}). For each $g \in \QS$,
we introduce the {\it quasisymmetry constant} of $g$ as
$$
M(g)=\sup_{0 \leq x < 1,\,0<t \leq 1/2} m_{\widetilde g}(x,t)^{\pm 1}.
$$
This defines a topology on $\QS$. More precisely, $g_n \in \QS$ converge to $g \in \QS$ if
$M(g_n \circ g^{-1}) \to 1$ as $n \to \infty$.
Then, the relative topology on $\QS_* \subset \QS$ 
coincides with the Teich\-m\"ul\-ler topology on $T \cong \QS_*$, which is the quotient topology under $\pi:\Bel(\D) \to T$
(see Theorem III.3.1 in Lehto \cite{Leh}).

We consider a special class of quasisymmetric functions on $\R$ 
whose quasisymmetry quotient tends to 1 uniformly as $t \to 0$.
We also consider the corresponding quasisymmetric homeomorphisms of $\S1$.

\begin{definition}
A quasisymmetric function $h:\R \to \R$ is called {\it symmetric} if there exists a non-negative increasing function
$\varepsilon(t)$ for $t>0$ with $\lim_{t \to 0}\varepsilon(t)=0$
such that 
$$
(1+\varepsilon(t))^{-1} \leq m_h(x,t) \leq 1+\varepsilon(t)
$$ 
for all $x \in \R$. We call $\varepsilon(t)$ a gauge function for symmetry.
A quasisymmetric homeomorphism $g \in \QS$ is called {\it symmetric} if its lift
$\widetilde g:\R \to \R$ is a symmetric function. We denote the subset of all symmetric self-homeomorphisms of $\S1$
by $\Sym \subset \QS$. 
\end{definition}

As the corresponding concept for quasiconformal maps, 
there are asymptotically conformal homeomorphisms whose complex dilatations vanish
at the boundary. We will review the relation of these two maps. In particular, we consider a certain quantitative estimate of 
the complex dilatation of the quasiconformal extension in terms of the quasisymmetry quotient. This
was originally studied by Carleson \cite{Car}. 

For a quasisymmetric function $h:\R \to \R$, we set
$$
\alpha(x,y)=\int_0^1 h(x+ty)dt; \qquad \beta(x,y)=\int_0^1 h(x-ty)dt,
$$
and define 
$$
F(z)=\frac{1}{2}[\alpha(x,y)+\beta(x,y)]+i[\alpha(x,y)-\beta(x,y)]
$$
for $z=x+iy \in \H2$. 
Beurling and Ahlfors \cite{BA} proved that $F$ is a quasiconformal self-homeomorphism of $\H2$
with an estimate of the maximal dilatation of $F$ in terms of the quasisymmetry constant $M \geq 1$ of $h$.
We call this the {\it Beurling--Ahlfors extension} of $h$. 
With regard to the Beurling--Ahlfors extension of symmetric functions, the following result, 
which was proved in Lemma 3 of \cite{Car} and 
improved slightly by providing an explicit computation for involved constants in Theorem 5.1 of \cite{Mat0}, is crucial.

\begin{theorem}\label{Carleson1}
Let $h:\R \to \R$ be a symmetric function such that $m_h(x,t)^{\pm 1} \leq 1+\varepsilon(t)$ 
for a gauge function $\varepsilon(t)$. Let $F$ be the Beurling--Ahlfors extension of $h$,
which is a quasiconformal self-homeomorphism of $\H2$. Then, the complex dilatation $\mu_F$ of $F$
satisfies
$|\mu_F(z)| \leq 4\varepsilon(y)$
for every $z=x+iy \in \H2$.
\end{theorem}

In particular, this theorem shows that a symmetric function $h:\R \to \R$ extends continuously to a
quasiconformal homeomorphism $F:\H2 \to \H2$ with $F(\infty)=\infty$ whose complex dilatation
$\mu_F(z)$ uniformly tends to $0$ as $y \to 0$ on $x \in \R$.

Conversely, such a quasiconformal self-homeomorphism $F$ of $\H2$ 
extends to a symmetric function on $\R$.  
Lemma 2 of Carleson \cite{Car} proved this fact, giving the order
of a gauge function for symmetry.
We will reprove this result in the following form with a more explicit estimate for the gauge function.
This estimate is useful in later arguments.

\begin{theorem}\label{ac-sym}
If a $K$-quasiconformal homeomorphism $F:\H2 \to \H2$ with $F(\infty)=\infty$ satisfies
$|\mu_F(z)| \leq \widetilde \varepsilon (y)$ uniformly on $x \in \R$ for a function $\widetilde \varepsilon(y)$ with 
$\widetilde \varepsilon(y) \to 0$ 
as $y \to 0$, then its boundary extension
$h:\R \to \R$ is a symmetric function whose quasisymmetry quotient satisfies
$m_h(x,t)^{\pm 1} \leq 1+\varepsilon(t)$ for
a gauge function 
$\varepsilon(t)$ with 
$$
\varepsilon(t) \leq c\, \widetilde \varepsilon(\sqrt{t})+R \sqrt{t} \quad (0< t \leq 1/2),
$$
where $c=c(K)>0$ is a constant depending only on $K \geq 1$, and $R>0$ is an absolute constant.
\end{theorem}

\begin{proof}
For each $t \in (0,1/2]$, 
we define a Beltrami coefficient $\mu_t(z)$ by letting $\mu_t(z)=\mu_F(z)$ on $\{z \in \H2 \mid y>\sqrt{t}\}$ and $\mu_t(z) \equiv 0$ elsewhere.
Let $F_t$ be the quasiconformal self-homeomorphism of $\H2$ with complex dilatation $\mu_t$ and with $F_t(\infty)=\infty$,
and $h_t$ the quasiconformal self-homeomorphism of $\H2$ such that $F=h_t \circ F_t$. As $h_t=F \circ F_t^{-1}$,
the complex dilatation of $h_t$ satisfies $|\mu_{h_t}(F_t(z))| \leq \widetilde \varepsilon(\sqrt{t})$ for
almost every $z \in \H2$.
In particular, there is a constant $c'>0$ depending only on $K$ such that
the maximal dilatation of $h_t$ is estimated as $K(h_t) \leq 1+c'\widetilde \varepsilon(\sqrt{t})$.

By reflection with respect to $\R$, we may assume that $F_t$ is a quasiconformal self-homeomorphism of $\C$.
The restriction of $F_t$ to the strip domain $\{z \in \C \mid |y|<\sqrt{t}\}$ is conformal.
For each $x \in \R$, we consider the ball of radius $\sqrt{t}$ with center $x$ and apply 
the Koebe distortion theorem (Proposition \ref{Koebe} below) to 
the conformal homeomorphism $F_t$ on this disk. Then, 
$$
\frac{|F_t'(x)|t}{(1+t/\sqrt{t})^2} \leq F_t(x+t)-F_t(x) \leq \frac{|F_t'(x)|t}{(1-t/\sqrt{t})^2}.
$$
The middle term can be replaced with $F_t(x)-F_t(x-t)$. 
This leads us to the following estimate for the quasisymmetry quotient $m_{F_t}(x,t)$ of $F_t|_{\R}$:
$$
\frac{(1-\sqrt{t})^2}{(1+\sqrt{t})^2}
\leq m_{F_t}(x,t)=\frac{F_t(x+t)-F_t(x)}{F_t(x)-F_t(x-t)} \leq \frac{(1+\sqrt{t})^2}{(1-\sqrt{t})^2}.
$$
In particular, there is an absolute constant $R'>0$ such that
$m_{F_t}(x,t)^{\pm 1} \leq 1+R'\sqrt{t}$ for $0<t \leq 1/2$.

Next, we apply the quasiconformal homeomorphism $h_t$ to the points $F_t(x-t)$, $F_t(x)$, and $F_t(x+t)$,
which are mapped to $h(x-t)$, $h(x)$, and $h(x+t)$, respectively.
We note that the quasisymmetry quotients can be given by
the conformal moduli as follows:
\begin{align*}
m_{F_t}(x,t)
&=\lambda(\modulus \H2(F_t(x-t),F_t(x),F_t(x+t),\infty));\\
m_{h}(x,t)
&=\lambda(\modulus \H2(h(x-t),h(x),h(x+t),\infty)).
\end{align*}
Here, $\modulus Q(x_1,x_2,x_3,x_4) \in (0,\infty)$ stands for the conformal modulus of 
a quadrilateral $Q$ with four positively ordered vertices $x_1, x_2, x_3, x_4 \in\partial Q$, and 
$\lambda:(0,\infty) \to (0,\infty)$ is the distortion function,
which transforms conformal moduli to quasisymmetry quotients
(see Section I.2.4 of \cite{Leh} and Section II.6 of \cite{LV}).

Moreover, the ratio of the conformal moduli are bounded by the maximal dilatation $K(h_t) \leq 1+c'\widetilde \varepsilon(\sqrt{t})$:
$$
\frac{1}{K(h_t)} \leq \frac{\modulus \H2(h(x-t),h(x),h(x+t),\infty)}{\modulus \H2(F_t(x-t),F_t(x),F_t(x+t),\infty)} \leq K(h_t).
$$
Plugging the quasisymmetry quotients in this inequality gives 
\begin{align*}
m_{h}(x,t)
&=\lambda(\modulus \H2(h(x-t),h(x),h(x+t),\infty))\\
&\leq \lambda(K(h_t) \modulus \H2(F_t(x-t),F_t(x),F_t(x+t),\infty))\\
&=\lambda(K(h_t) \lambda^{-1} (m_{F_t}(x,t)))\\
&\leq \lambda((1+c'\widetilde \varepsilon(\sqrt{t})) \lambda^{-1} (1+R'\sqrt{t}))
\end{align*}
for all $x \in \R$ and $t \in (0,1/2]$.
An estimate for $m_{h}(x,t)^{-1}$ is similarly obtained.
Because $\lambda$ is continuous and increasing with $\lambda(1)=1$ and
differentiable at $1$ with a non-vanishing derivative (see e.g. \cite{AVV}), we see that
the last term can be represented as $1+\varepsilon(t)$ for
a gauge function 
$\varepsilon(t)$ as in the statement of the theorem.
\end{proof}

We review the Koebe distortion theorem, which 
includes the one-quarter theorem
(see Theorem 1.3 in \cite{P}).

\begin{proposition}\label{Koebe}
A conformal homeomorphism $f$ of $\D$ into $\C$ satisfies
\begin{align*}
&\quad |f'(0)|\frac{|z|}{(1+|z|)^2} \leq |f(z)-f(0)| \leq |f'(0)|\frac{|z|}{(1-|z|)^2};\\
&\quad |f'(0)|\frac{1-|z|}{(1+|z|)^3} \leq |f'(z)| \leq |f'(0)|\frac{1+|z|}{(1-|z|)^3}
\end{align*}
for every $z \in \D$. The first inequality in the former line in particular shows that the image $f(\D)$ contains a disk
with its center at $f(0)$ and radius $|f'(0)|/4$.
\end{proposition}

Using the Beurling--Ahlfors extension, we can also define a quasiconformal extension of
a quasisymmetric self-homeomorphism $g$ of $\S1$ to $\D$. Actually, for the lift $\widetilde g:\R \to \R$ of $g$
under the universal cover $u:\R \to \S1$, we take the Beurling--Ahlfors extension $F:\H2 \to \H2$ of $\widetilde g$.
Here, we also use the extension of $u$ to the holomorphic universal cover $u:\H2 \to \D-\{0\}$
defined by $u(z)=e^{2\pi iz}$.
By projecting down $F$ to a quasiconformal self-homeomorphism of $\D-\{0\}$ by the holomorphic universal cover $u$ and
filling the puncture $0$, we obtain a quasiconformal self-homeomorphism $f$ of $\D$. 
By this correspondence $g \mapsto f$, we have a map 
$$
e_{\rm BA}:\QS \to \QC(\D),
$$
which satisfies $q \circ e_{\rm BA}=\id|_{\QS}$.

Unlike the Douady--Earle extension $e_{\rm DE}$, the Beurling--Ahlfors extension $e_{\rm BA}$
does not have conformal naturality. Accordingly, it does not descend to a section $T \to \Bel(\D)$ naturally.
In order to define a section, we use the normalized quasisymmetric homeomorphism $g \in \QS_*$
as a representative of an element $[g] \in T$. From this $g$, 
we make the quasiconformal self-homeomorphism $f$ of $\D$ 
as above, and then take its complex dilatation $\mu_f$. By this correspondence $[g] \mapsto \mu_f$, we have a map
$s_{BA}:T \to \Bel(\D)$, which is a section for the Teich\-m\"ul\-ler projection $\pi:\Bel(\D) \to T$.
It can also be proved that $s_{BA}$ is continuous.

We say that
a quasiconformal homeomorphism $f \in \QC(\D)$ is {\it asymptotically conformal} if
the complex dilatation $\mu_{f}(z)$ vanishes at the boundary $\S1$. This means that
$$
\lim_{t \to 0} \esssup\{\,|\mu_{f}(z)| \mid |z| \geq 1-t\}=0.
$$
We denote the subset of $\QC(\D)$ consisting of all asymptotically conformal homeomorphisms 
by $\AC(\D)$. Theorem \ref{Carleson1} implies that the restriction of $e_{\rm BA}$ to $\Sym$ gives
$$
e_{\rm BA}:\Sym \to \AC(\D).
$$
Moreover, Theorem \ref{ac-sym} implies that the restriction of $q$ to $\AC(\D)$ gives
$$
q:\AC(\D) \to \Sym.
$$
We note that, for a given point $z_0$ in $\D$, there is a quasiconformal self-homeo\-morphism $\phi$ 
with $\phi(z_0)=0$ and $q(\phi)=\id|_{\S1}$ whose complex dilatation vanishes outside some compact subset in $\D$.
The composition of such a map $\phi$ makes any asymptotically conformal self-homeomorphism of $\D$ fix $0$ 
without changing the property of vanishing at the boundary. 

By the above two claims, we have the following result attributed to Fehlmann \cite{Feh} in Gardiner and Sullivan \cite{GS}:

\begin{corollary}\label{acsym}
A quasisymmetric homeomorphism $g$ is in $\Sym$ if and only if
$g$ extends continuously to a quasiconformal homeomorphism in $\AC(\D)$.
\end{corollary}

By the chain rule of complex dilatations, the composition of asymptotically conformal
self-homeomorphisms of $\D$ is also asymptotically conformal. Hence, $\AC(\D)$ is a
subgroup of $\QC(\D)$. Accordingly, Corollary \ref{acsym} shows that
$\Sym$ is a subgroup of $\QS$.
Moreover, it was proved in \cite{GS} that $\Sym$ is the characteristic topological subgroup of 
the partial topological group $\QS$ for which the neighborhood base is given at $\id$ by using the quasisymmetry constant and
is distributed at every point $g \in \QS$ by the right translation.

In the rest of this section,
we review 
the Teich\-m\"ul\-ler space of symmetric homeomorphisms, which is already well-known in
the theory of asymptotic Teich\-m\"ul\-ler spaces. This will be a prototype of
our construction of the Teich\-m\"ul\-ler space of circle diffeomorphisms.

\begin{definition}
The {\it little subspace} $T_0$ 
of the universal Teich\-m\"ul\-ler space $T$
(or the Teich\-m\"ul\-ler space of symmetric homeomorphisms) 
is defined as 
$$
T_0=\Mob(\S1) \backslash \Sym \subset T=\Mob(\S1) \backslash \QS.
$$
\end{definition}

We define the subset $\Bel_0(\D)$ of $\Bel(\D)$ consisting of all
Beltrami coefficients vanishing at the boundary. As $\Mob(\D) \backslash \AC(\D)$ can be identified with
$\Bel_0(\D)$, Corollary \ref{acsym} implies that the image of $\Bel_0(\D)$ under
the Teich\-m\"ul\-ler projection $\pi:\Bel(\D) \to T$ is $T_0$. 
This also implies that its Bers embedding 
$\beta(T_0)$ coincides with $\Phi(\Bel_0(\D))$ for the Bers projection $\Phi:\Bel(\D) \to B(\D^*)$.
Under the group structure $\ast$ of $\Bel(\D)$, $\Bel_0(\D)$ is a subgroup.
Correspondingly, $T_0$ is a subgroup of $(T,\ast)$.
In fact, $T_0 \subset T$ is a topological subgroup as $T_0 \cong \Sym \cap \QS_*$ and $\Sym$ is a topological subgroup.

It was proved by Earle, Markovic, and Saric \cite{EMS} that the Douady--Earle extension $e_{\rm DE}(g)$ of
a symmetric homeomorphism $g \in \Sym$ is asymptotically conformal;
$e_{\rm DE}:\Sym \to \AC(\D)$ is a section of $q:\AC(\D) \to \Sym$.
Hence, the conformally natural section $s_{\rm DE}:T \to \Bel(\D)$ sends $T_0$ to $\Bel_0(\D)$.
We note that $\Bel_0(\D)$ is the unit ball of the Banach subspace $L_0^\infty(\D) \subset L^\infty(\D)$
consisting of bounded measurable functions vanishing at the boundary:
$\Bel_0(\D)=\Bel(\D) \cap L_0^\infty(\D)$. In particular, $\Bel_0(\D)$ is contractible.
Therefore, $T_0$ is also contractible. 

To consider the complex structure of $T_0$, we introduce the Banach subspace
$B_0(\D^*)$ of $B(\D^*)$ as follows:
$$
B_0(\D^*)=\{\varphi \in B(\D^*) \mid \lim_{|z| \to 1} \rho_{\D^*}^{-2}(z)|\varphi(z)|=0\}.
$$
An element in $B_0(\D^*)$ is also called vanishing at the boundary.
The following theorem, which was
essentially proved by
Becker and Pommerenke \cite{BP}, can be found in \cite{GS}. 

\begin{theorem}\label{B_0}
For the Bers projection $\Phi:\Bel(\D) \to B(\D^*)$,
it holds that
$$
\Phi(\Bel_0(\D))=\beta(T) \cap B_0(\D^*).
$$ 
\end{theorem}

By this theorem, we have $\beta(T_0)=\beta(T) \cap B_0(\D^*)$. Hence, $T_0$ is identified with
a bounded contractible domain of the complex Banach space $B_0(\D^*)$.

\section{The decay order of Schwarzian and pre-Schwarzian derivatives}\label{4}

We focus on the decay order of a Beltrami coefficient $\mu \in \Bel_0(\D)$ vanishing at the boundary $\S1$.
We define
$$
\kappa_\mu(t)=\esssup_{1-t \leq |\zeta| <1} |\mu(\zeta)| \qquad (0<t \leq 1)
$$
for $\mu \in \Bel_0(\D)$, which satisfies 
$\kappa_\mu(t) \to 0$ as $t \to 0$.
Let $\alpha \in (0,1)$ be a fixed constant.
For a Beltrami coefficient $\mu \in \Bel_0(\D)$, we define 
a new norm by
$$
\Vert \mu \Vert_{\infty,\alpha}=\esssup_{\zeta \in \D} \rho^{\alpha}_{\D}(\zeta)|\mu(\zeta)|.
$$
Clearly, $\Vert \mu \Vert_{\infty,\alpha}<\infty$ if and only if $\kappa_\mu(t)=O(t^\alpha)$ $(t \to 0)$.

\begin{definition}
Let $\alpha$ be a constant with $0<\alpha<1$.
The space of Beltrami coefficients $ \mu \in \Bel(\D)$ with $\Vert \mu \Vert_{\infty,\alpha}<\infty$
is denoted by $\Bel_0^{\alpha}(\D)$.
\end{definition}

As in the definition of the Bers projection,
we extend a Beltrami coefficient $\mu \in \Bel^\alpha_0(\D)$ to $\Chat$
by setting $\mu(z) \equiv 0$ for $z \in \D^*$ and take 
a quasiconformal homeomorphism $f_\mu: \Chat \to \Chat$
having the complex dilatation $\mu$. Then,
$f_\mu|_{\D^*}$ is a conformal homeomorphism (univalent function).
Hereafter, we always give the following normalization for $f_\mu$:
$$
f_\mu(\infty)=\infty; \quad \lim_{z \to \infty} f_\mu'(z)=1.
$$
Equivalently, the Laurent expansion of $f_\nu$ at $\infty$ is
$$
f_\mu(z)=z+b_0+\frac{b_1}{z}+\cdots.
$$
We consider its pre-Schwarzian derivative 
and Schwarzian derivative on $\D^*$, which are
defined respectively as follows:
\begin{align*}
T_{f_\mu|_{\D^*}}(z)&=\frac{f_\mu''(z)}{f_\mu'(z)}; \\
S_{f_\mu|_{\D^*}}(z)&=\left(T_{f_\mu|_{\D^*}}\right)'(z)
-\frac{1}{2}\left(T_{f_\mu|_{\D^*}}(z)\right)^2.
\end{align*}

It was shown in Becker and Pommerenke \cite{BP} that the condition 
$\mu \in \Bel_0(\D)$  
is equivalent to each of the conditions
$$
\lim_{|z| \to 1} \rho_{\D^*}^{-1}(z)|T_{f_\mu|_{\D^*}}(z)|=0; \qquad
\lim_{|z| \to 1} \rho_{\D^*}^{-2}(z)|S_{f_\mu|_{\D^*}}(z)|=0.
$$
To estimate their decay order quantitatively in terms of $\kappa_\mu(t)$, we set
\begin{align*}
\beta_\mu(t)&=\max_{|z| = 1+t}(|z|-1)|T_{f_\mu|_{\D^*}}(z)|,\\ 
\sigma_\mu(t)&=\max_{|z| = 1+t}(|z|-1)^2|S_{f_\mu|_{\D^*}}(z)| \qquad (0<t<\infty).
\end{align*}
It was proved in Theorem 2 of Becker \cite{Bec2} that 
$$
\beta_\mu(t^{1+\varepsilon}) \leq 3 (\kappa_\mu(t)+t^{\varepsilon}),\quad
\sigma_\mu(t^{1+\varepsilon}) \leq \frac{3}{2} (\kappa_\mu(t)+t^{2\varepsilon}) \qquad (0<t \leq 1)
$$
for any $\varepsilon >0$. We note that the above definitions of $\beta_\mu$ and $\sigma_\mu$
are slightly different from those in \cite{Bec2}.

We will improve these estimates regarding the power of $t$
for the case where $\kappa_\mu(t)=O(t^\alpha)$ $(t \to 0)$.
In this case, the elimination of the constant $\varepsilon$ was done by Dyn$'$kin \cite{Dy}.
Our improvement can be stated as follows.

\begin{theorem}\label{basic1}
For every $\alpha \in (0,1)$, there is a constant $C=C(\alpha)>0$ that depends only on $\alpha$
such that
$$
\rho_{\D^*}^{-1}(z)|T_{f_\mu|_{\D^*}}(z)| \leq C \Vert \mu \Vert_{\infty,\alpha} (|z|-1)^\alpha
$$
for every $\mu \in \Bel^\alpha_0(\D)$ and for every $z \in \D^*$.
Equivalently,
$$
\beta_\mu(t) \leq C \Vert \mu \Vert_{\infty,\alpha} \frac{2t^\alpha}{t+2}
$$ 
for every $t>0$.
\end{theorem}
 
We decompose a Beltrami coefficient $\mu \in \Bel_0^\alpha(\D)$
suitably into a finite number of Beltrami coefficients whose supports are in mutually disjoint annular domains of $\D$.
Then, a computation of the pre-Schwarzian derivative of the composition of the corresponding conformal homeomorphisms
establishes the estimate.
These steps are given in the following two lemmata.

\begin{lemma}\label{recurrence}
For every $\alpha \in (0,1)$,
there is a constant $\lambda$ with $0 < \lambda <1$ that depends only on $\alpha$ such that,
if a sequence $\{s_n\}_{n=0}^\infty$ of positive numbers satisfies a recurrence relation
$$
\left(\frac{1}{1+s_{n-1}} \right) {s_n}^\alpha=\lambda^n 
$$
for every $n \geq 1$ and $s_0=1$, 
then $\{s_n\}$ is increasing and diverges to $+\infty$.
\end{lemma}

\begin{proof}
The recurrence relation is equivalent to
$$
s_n=\lambda^{\frac{n}{\alpha}}(1+s_{n-1})^{\frac{1}{\alpha}}
$$
for every $n \geq 1$ and $s_0=1$. For comparison with this formula, we consider another recurrence relation
$$
s'_n=\lambda^{\frac{n}{\alpha}}{s'_{n-1}}^{\frac{1}{\alpha}}
$$
for every $n \geq 2$ by giving the initial value $s'_1=s_1=(2\lambda)^{1/\alpha}$.
It is easy to see that $s_n \geq s'_n$ for every $n \geq 1$, and hence,
$\lim_{n \to \infty} s'_n=+\infty$ implies $\lim_{n \to \infty} s_n=+\infty$.
Moreover, if $\{s'_n\}$ is increasing then so is $\{s_n\}$.

Let $b_n=s'_{n+1}/s'_n$. Then, we have
$$
b_n=\lambda^{\frac{1}{\alpha}} (b_{n-1})^{\frac{1}{\alpha}}
$$
for every $n \geq 2$ and 
$$
b_1=\frac{s'_2}{s'_1}=\frac{\lambda^{\frac{2}{\alpha}}(2\lambda)^{\frac{1}{\alpha^2}}}{(2\lambda)^{\frac{1}{\alpha}}}.
$$
Taking the logarithm yields
$$
\log b_n=\frac{1}{\alpha} \log b_{n-1} +\frac{1}{\alpha} \log\lambda
$$
with
$$
\log b_1=\left(\frac{1}{\alpha^2}+\frac{1}{\alpha}\right) \log \lambda+
\left(\frac{1}{\alpha^2}-\frac{1}{\alpha}\right)\log 2.
$$
This shows that if
$$
\log b_1 > \frac{-\log \lambda}{1-\alpha},
$$
then $\log b_n$ are positive and uniformly bounded away from $0$ for all $n \geq 1$.
By choosing $\lambda<1$ such that it is sufficiently close to $1$, we have such a situation.
For instance, $\lambda$ can be chosen so that
$\lambda>(1/2)^{(1-\alpha)^2/(1+\alpha+\alpha^2)}$.
This proves that $\{s'_n\}$ is increasing and diverges to $+\infty$.
\end{proof}

\begin{lemma}\label{decomposition}
For a finite sequence of real numbers
$$
1=r_{-1}>r_0>r_1> \cdots >r_{N}>r_{N+1}=0,
$$
let $A_n=\{\zeta \in \D \mid r_n > |\zeta| \geq r_{n+1}\}$ be an annulus (or a disk) in $\D$ for each $n=-1,0,\ldots,N$.
For any $\mu \in \Bel(\D)$ and each $n$,
we define a Beltrami coefficient on $\Chat$ by
$$
\mu_n(\zeta)=\left\{
\begin{array}{l} \mu(\zeta) \,\quad (\zeta \in A_n)\\
\ 0  \qquad (\zeta \in \Chat-A_n).
\end{array}
\right.
$$ 
Let $k_n=\Vert \mu_n \Vert_\infty$.
Then, the pre-Schwarzian derivative of $f_{\mu}|_{\D^*}$ satisfies 
$$
|T_{f_{\mu}|_{\D^*}}(z)| \leq 12 \sum_{n=-1}^{N} \frac{k_n r_n}{|z|^2-r_n^2} 
$$
for every $z \in \D^*$.
\end{lemma}

\begin{proof}
First, we take a quasiconformal self-homeomorphism $f_{N}$ of 
$\C$ (namely, that 
of $\Chat$ fixing $\infty$) having the complex dilatation $\mu_{N}$, and 
consider the push-forward $\widetilde \mu_{N-1}=(f_{N})_* \mu_{N-1}$
of $\mu_{N-1}$ by $f_{N}$, which is conformal on $A_{N-1}$. 
Here, the push-forward $f_*\mu$ of $\mu \in \Bel(D)$ by a conformal homeomorphism $f$ 
of a domain $D$ is defined in general by
$$
(f_*\mu)(z)=\mu(f^{-1}(z)) \frac{\overline{(f^{-1})'(z)}}{(f^{-1})'(z)} \qquad (z \in f(D)).
$$
Next, we take a quasiconformal self-homeomorphism 
$f_{N-1}$ of $\C$ 
having the complex dilatation $\widetilde \mu_{N-1}$
and the push-forward 
$\widetilde \mu_{N-2}=(f_{N-1}\circ f_{N})_* \mu_{N-2}$.
Inductively, for each $n \geq 0$,
let $f_{n}$ be a quasiconformal self-homeomorphism of $\C$ whose complex dilatation is
$\widetilde \mu_{n}$ and let 
$$
\widetilde \mu_{n-1}=(f_{n}\circ \cdots \circ f_{N})_* \mu_{n-1}
$$ 
be the push-forward 
of $\mu_{n-1}$ by $f_{n}\circ \cdots \circ f_{N}$.
Finally, we choose a quasiconformal self-homeomorphism $f_{-1}$ of $\C$ 
with the complex dilatation $\widetilde \mu_{-1}$
so that 
$f_{-1} \circ \cdots \circ f_{N}$ coincides with $f_{\mu}$.

By the chain rule of pre-Schwarzian derivatives, we see that
\begin{align*}
&\quad \ T_{f_{\mu}|_{\D^*}}(z)\\
&=T_{f_{N}}(z)+T_{f_{N-1}}(f_N(z))f'_N(z)+ \cdots +
T_{f_{-1}}(f_{0} \circ \cdots \circ f_{N}(z))(f_{0} \circ \cdots \circ f_{N})'(z)\\
&=
T_{f_N}(z)+\sum_{n=-1}^{N-1} T_{f_n}(f_{n+1} \circ \cdots \circ f_{N}(z))(f_{n+1} \circ \cdots \circ f_{N})'(z)
\end{align*}
for every $z \in \D^*$.

Here, we use the following estimates for the pre-Schwarzian derivative.
For any conformal homeomorphism $f$ of $\D^*$ with $f(\infty)=\infty$, it was shown in Avhadiev \cite{Av}
(cf. Theorem 4.2.3 in Sugawa \cite{Sug}) that
$$
\rho_{\D^*}^{-1}(z)|T_f(z)| \leq \frac{|z|^2-1}{2}|z\, T_f(z)| \leq 3 \qquad (|z|>1).
$$
In addition, if $f$ extends to a quasiconformal self-homeomorphism of $\Chat$ 
of complex dilatation $\mu$ with $\Vert \mu \Vert_\infty \leq k$,
then the majorant principle as described in Section II.3.5 of Lehto \cite{Leh} 
yields that $|T_f(z)| \leq 3k \rho_{\D^*}(z)$. Moreover,
for any simply connected domain $\Omega^* \subset \Chat$ containing $\infty$ and for any
conformal homeomorphism $f$ of $\Omega^*$ with $f(\infty)=\infty$, we see that
$|T_f(\omega)| \leq 6\rho_{\Omega^*}(\omega)$ for $\omega \in \Omega^*$, where $\rho_{\Omega^*}(\omega)$ is the
hyperbolic density on $\Omega^*$. This stems from the chain rule of pre-Schwarzian derivatives and the invariance of a hyperbolic metric
(see Theorem 1 in Osgood \cite{Os}). Again, if this extends to a quasiconformal homeomorphism of $\Chat$ 
with $\Vert \mu \Vert_\infty \leq k$, then $|T_f(\omega)| \leq 6k\rho_{\Omega^*}(\omega)$.

The conformal homeomorphism $f_{N}$ of the disk 
$\Omega_{N}^*=\{|z|>r_{N}\} \cup \{\infty\}$ into $\Chat$ with $f_N(\infty)=\infty$ satisfies 
$$
|T_{f_{N}}(z)|
\leq \frac{6k_N r_N}{|z|^2-r_N^2}.
$$
The conformal homeomorphism $f_{n}$ of the 
quasidisk $\Omega_n^*$ into $\Chat$ with $f_n(\infty)=\infty$ for $-1 \leq n \leq N-1$, 
where $\Omega_n^*$ is the image of the disk $\{|z|>r_n\} \cup \{\infty\}$ under
$f_{n+1} \circ \cdots \circ f_{N}$, satisfies 
$$
|T_{f_n}(\omega)| \leq 6k_n \rho_{\Omega_n^*}(\omega)
$$
for every $\omega \in \Omega_n^*$ in terms of the hyperbolic density $\rho_{\Omega_n^*}(\omega)$ of $\Omega_n^*$.
Hence, by replacing $\omega$ with $f_{n+1} \circ \cdots \circ f_{N}(z)$, we obtain
\begin{align*}
&\quad |T_{f_n}(f_{n+1} \circ \cdots \circ f_{N}(z))(f_{n+1} \circ \cdots \circ f_{N})'(z)|\\
&\leq 6 k_n \rho_{\Omega_n^*}(f_{n+1} \circ \cdots \circ f_{N}(z))|(f_{n+1} \circ \cdots \circ f_{N})'(z)|
=
\frac{12k_n r_n}{|z|^2-r_n^2}. 
\end{align*}
This gives the desired inequality
$$
|T_{f_{\mu}|_{\D^*}}(z)| \leq 12 \sum_{n=-1}^{N} \frac{k_n r_n}{|z|^2-r_n^2}
$$
for every $z \in \D^*$.
\end{proof}

\medskip
\noindent
{\it Proof of Theorem \ref{basic1}.\ }
For any $\mu \in \Bel_0^\alpha(\D)$, let $\ell=\Vert \mu \Vert_{\infty,\alpha}<\infty$.
Then, 
$$
\kappa_{\mu}(t)=\sup_{1-t \leq |\zeta| < 1} |\mu(\zeta)| \leq \ell t^\alpha \qquad (0<t \leq 1).
$$
Fixing $z \in \D^*$, we will estimate $\rho_{\D^*}^{-1}(z)|T_{f_{\mu}|_{\D^*}}(z)|$
in terms of $\ell$.
In the case of $|z| \geq 2$, we can easily obtain the desired estimate. Indeed, by the inequality
$\rho_{\D^*}^{-1}(z)|T_{f_{\mu}|_{\D^*}}(z)| \leq 3 \Vert \mu \Vert_\infty$
as in the proof of Lemma \ref{decomposition} and by $\Vert \mu \Vert_{\infty} \leq \Vert \mu \Vert_{\infty,\alpha}$, we obtain
$$
\rho_{\D^*}^{-1}(z)|T_{f_{\mu}|_{\D^*}}(z)| \leq 3 \Vert \mu \Vert_{\infty,\alpha} \leq 3 \Vert \mu \Vert_{\infty,\alpha}(|z|-1)^\alpha.
$$
Hence, we may assume that $1<|z|<2$. Let $\tau=|z|-1 \in (0,1)$.

We choose $t_0=\tau$ and inductively define a sequence $\{t_n\}_{n \geq 1}$ of positive numbers by
a recurrence relation
$$
\frac{\tau}{\tau+t_{n-1}}\cdot \ell {t_n}^\alpha=\lambda^n \cdot \ell \tau^\alpha
$$
for some constant $\lambda$ with $0<\lambda<1$. If we set $s_n=t_n/\tau$, this is equivalent to
$$
\left(\frac{1}{1+s_{n-1}} \right) {s_n}^\alpha=\lambda^n
$$
with the initial condition $s_0=1$. Then, by Lemma \ref{recurrence},
we can find the constant $\lambda=\lambda(\alpha)$ so that the sequence $\{s_n\}$, and hence $\{t_n\}$
are increasing and diverge to $+\infty$.
In particular, there is the smallest non-negative integer $N \geq 0$ such that 
$t_{N+1} \geq 1$.

By using the positive numbers $\{t_n\}_{n=0}^N$, 
we set $r_n=1-t_n$. We also set $r_{-1}=1$ 
and $r_{N+1}=0$.
Then, as in Lemma \ref{decomposition}, we divide $\D$ into
the annuli (or the disk)
$$
A_n=\{\zeta \in \D \mid r_n > |\zeta| \geq r_{n+1}\} \qquad (n=-1,0,\ldots ,N)
$$ 
and define $k_n=\Vert \mu_n \Vert_\infty$
for $\mu_n=\mu \cdot 1_{A_n}$.  
Because $\kappa_{\mu}(t) \leq \ell t^\alpha$,
we see that $k_n \leq \ell {t_{n+1}}^\alpha$.
We note that for $n=N$, this is valid as $\Vert \mu \Vert_\infty \leq \ell \leq \ell {t_{n+1}}^\alpha$.
Now, the application of Lemma \ref{decomposition} yields 
$$
\rho_{\D^*}^{-1}(z)|T_{f_{\mu}|_{\D^*}}(z)| \leq 6 (|z|^2-1) \sum_{n=-1}^{N} \frac{k_n r_n}{|z|^2-r_n^2}
\leq
6 \sum_{n=-1}^{N} \frac{\tau}{\tau+t_n}\cdot \ell {t_{n+1}}^\alpha.
$$
Here, the recurrence relation for $\{t_n\}$ shows that 
the last sum is taken for $\lambda^{n+1} \cdot \ell \tau^\alpha$.
Thus, 
$$
\rho_{\D^*}^{-1}(z)|T_{f_{\mu}|_{\D^*}}(z)| \leq \frac{6 \ell}{1-\lambda} \tau^\alpha,
$$
where $\lambda$ depends only on $\alpha$.
By taking $C=6/(1-\lambda)$, we obtain the desired inequality.
\qed
\medskip

Next, we consider the relation between $T_f$ and $S_f$ for a conformal homeomorphism $f$ of $\D^*$.
It is known that there is some absolute constant $A>0$ such that
$$
\rho_{\D^*}^{-2}(z)|S_f(z)| \leq A \rho_{\D^*}^{-1}(z)|z\,T_f(z)| \qquad (z \in \D^*).
$$
(see Lemma 6.1 in Becker \cite{Bec1}). This in particular implies the following:

\begin{proposition}\label{T->S}
If $\beta_\mu(t)=O(t^\alpha)$ then $\sigma_\mu(t)=O(t^\alpha)$ $(t \to 0)$.
\end{proposition}

\begin{remark}
Lemmata \ref{recurrence} and \ref{decomposition} can be easily modified
so that they are suitable for estimation of Schwarzian derivatives. Hence, inequalities 
$$
\rho_{\D^*}^{-2}(z)|S_{f_\mu|_{\D^*}}(z)| \leq C' \Vert \mu \Vert_{\infty,\alpha} (|z|-1)^\alpha; \quad
\sigma_\mu(t) \leq C' \Vert \mu \Vert_{\infty,\alpha} \frac{4t^\alpha}{(t+2)^2}
$$
for some $C'=C'(\alpha)>0$
can also be derived directly
from these modifications in the same way as in the proof of Theorem \ref{basic1}. The condition 
$\sigma_\mu(t)=O(t^\alpha)$ $(t \to 0)$ is equivalent to 
$\sup_{z \in \D^*} \rho_{\D^*}^{-2+\alpha}(z)|S_{f_\mu|_{\D^*}}(z)|<\infty$.
\end{remark}

Finally, we will show that $\sigma_\mu(t)=O(t^\alpha)$ implies that $\kappa_{\mu'}(t)=O(t^\alpha)$ $(t \to 0)$
for some $\mu' \in \Bel(\D)$ with $\pi(\mu)=\pi(\mu')$.
This is a consequence of
the next lemma, which can be found in Theorem 5.4 of Becker \cite{Bec1}. We note that the condition $\pi(\mu)=\pi(\mu')$ 
is equivalent to $f_\mu|_{\D^*}=f_{\mu'}|_{\D^*}$ for $\mu, \mu' \in \Bel(\D)$.

\begin{lemma}\label{localext}
Let $f$ be a conformal homeo\-morphism of $\D^*$ having a quasiconformal extension to $\Chat$
such that $\varphi=S_{f}$ belongs to $B_0(\D^*)$.
We set
$$
F(z)=f(z^*)-\frac{(z^*-z)f'(z^*)}{1+(z^*-z)f''(z^*)/(2f'(z^*))}
$$
for $z \in \D$, where $z^*=1/\bar z$ is the reflection of $z$ with respect to $\S1$.
Then, there is some $t>0$ such that $f$ extends to a quasiconformal self-homeomorphism of $\Chat$
that coincides with $F$ on the annulus $\{1-t<|z|<1\}$ having the complex dilatation 
$$
\mu_F(z)=\frac{{\bar \partial}F(z)}{\partial F(z)}=-2 \rho^{-2}_{\D^*}(z^*)(zz^*)^2 \varphi(z^*).
$$
\end{lemma}

Theorem \ref{basic1}, Proposition \ref{T->S}, and Lemma \ref{localext} conclude
the equivalence of all the conditions above. 

\begin{theorem}\label{equivalence1}
The following conditions are equivalent for $\mu \in \Bel(\D)$ and $\alpha \in (0,1)$:
\begin{enumerate}
\item 
$\kappa_{\mu'}(t)=O(t^\alpha)$ $(t \to 0)$ for some $\mu' \in \Bel(\D)$ with $\pi(\mu)=\pi(\mu')$;
\item
$\beta_\mu(t)=O(t^\alpha)$ $(t \to 0)$;
\item
$\sigma_\mu(t)=O(t^\alpha)$ $(t \to 0)$. 
\end{enumerate}
\end{theorem}

The above results can also be proved when we exchange the role of $\D$ and $\D^*$.
We will briefly mention this fact.
For any Beltrami coefficient $\mu \in \Bel(\D)$,
we define its reflection by
$$
\mu^*(z)=\overline{\mu(z^*)}(zz^*)^2 \in \Bel(\D^*).
$$
This coincides with the complex dilatation of the reflection of
$f^\mu:\D \to \D$ with respect to $\S1$. If $\mu \in \Bel_0^\alpha(\D)$ and $\Vert \mu \Vert_{\infty,\alpha}=\ell<\infty$,
then $\mu^*$ satisfies
\begin{align*}
|\mu^*(z)|=|\mu(z^*)| \leq \ell \, \left (\frac{|z|^2-1}{2|z|^2} \right)^\alpha \leq \ell \, (|z|-1)^\alpha \qquad (z \in \D^*);\\
\kappa_{\mu^*}(t)=\sup_{1 < |z| \leq 1+t} |\mu^*(z)| \leq \ell t^\alpha \qquad (0<t <\infty).
\end{align*}

The function $\beta$ for 
$\mu^* \in \Bel(\D^*)$ is given similarly. 
We extend $\mu^*$ to $\Chat$
by setting $\mu^*(\zeta) \equiv 0$ for $\zeta \in \D$ and take 
a quasiconformal homeomorphism $f_{\mu^*}: \Chat \to \Chat$
having the complex dilatation $\mu^*$ with $f_{\mu^*}(\infty)=\infty$. Then,
for the pre-Schwarzian derivative
$T_{f_{\mu^*}|_{\D}}(\zeta)=f_{\mu^*}''(\zeta)/f_{\mu^*}'(\zeta)$
on $\D$, we define
$$
\bar \beta_{\mu^*}(t)=\max_{|\zeta|=1-t}(1-|\zeta|)|T_{f_{\mu^*}|_{\D}}(\zeta)| \qquad (0<t \leq 1).
$$

We can modify Lemma \ref{decomposition} appropriately by using the corresponding estimates of pre-Schwarzian derivatives
on $\D$ and any simply connected domain $\Omega \subset \C$:
$$
|T_f(\zeta)| \leq 3\rho_{\D}(\zeta) \quad (\zeta \in \D)\ ; \quad |T_f(\omega)| \leq 4\rho_{\Omega}(\omega) \quad (\omega \in \Omega).
$$
Concerning the relation between $T_f$ and $S_f$ for a conformal homeomorphism $f$ of $\D$,
there is some absolute constant $A'>0$ such that 
$$
\rho_{\D}^{-2}(\zeta)|S_f(\zeta)| \leq A' \rho_{\D}^{-1}(\zeta)|T_f(\zeta)| \qquad (\zeta \in \D).
$$
(see pp.117--119 of \cite{Bec2} and Sections 4.2 and 5.3 of \cite{Sug}).
Thus, the statement corresponding to Proposition \ref{T->S} holds true also in this case.
Moreover, the interior version of Lemma \ref{localext} is given in Theorem 3 of \cite{Bec2}.

Therefore, the statements that correspond to Theorems \ref{basic1} and \ref{equivalence1} are also valid in this case;
in particular, we record the following claim as a corollary for later use.

\begin{corollary}\label{interior}
For every $\alpha \in (0,1)$, there is a constant $C'=C'(\alpha)>0$ depending only on $\alpha$
such that
$\bar \beta_{\mu^*}(t) \leq C' \Vert \mu \Vert_{\infty,\alpha} t^\alpha$ 
for every $\mu \in \Bel^\alpha_0(\D)$ and for every $t \in (0,1]$.
\end{corollary}

\section{H\"older continuity of derivatives and quasisymmetry quotients}\label{5}

We define a class of orientation-preserving diffeomorphisms of the circle with H\"older continuous derivatives,
which is of importance in our theory of Teichm\"uller spaces.
In this section, we
investigate the topology of the space of such circle diffeomorphisms.
In particular, we relate this topology to the quasisymmetry quotients and the dilatations of their quasiconformal extensions.

\begin{definition}
An orientation-preserving diffeomorphism $g:\S1 \to \S1$ belongs to the class
$\Diff_+^{1+\alpha}(\S1)$ for exponent $\alpha \in (0,1)$ if its derivative is
$\alpha$-H\"older continuous. This means that the lift $\widetilde g:\R \to \R$ of $g$ under the universal cover $\R \to \S1$ satisfies
$$
|\widetilde g'(x)-\widetilde g'(y)| \leq c|x-y|^\alpha \qquad (x,\, y \in \R)
$$
for some $c \geq 0$.
\end{definition}

We provide the right uniform topology for $\Diff_+^{1+\alpha}(\S1)$. This is induced by
the $C^{1+\alpha}$-modulus $p_{1+\alpha}$, which measures the difference between
an element $g \in \Diff_+^{1+\alpha}(\S1)$ and the identity as follows:
$$
p_{1+\alpha}(g)=\sup_{\xi \in \S1} |g(\xi)-\xi|
+\sup_{0 \leq x <1} |\widetilde g'(x)-1|
+c_\alpha(g),
$$
where 
$$
c_\alpha(g)=\sup_{0<|x-y| \leq 1/2} \frac{|\widetilde g'(x)-\widetilde g'(y)|}{|x-y|^\alpha}.
$$
Then, $g_n$ converge to $g$ in $\Diff_+^{1+\alpha}(\S1)$ by definition if 
$p_{1+\alpha}(g_n \circ g^{-1}) \to 0$ as $n \to \infty$.

\begin{remark}
The right uniform topology on $\Diff_+^{1+\alpha}(\S1)$ as above is different from the $C^{1+\alpha}$-topology
given in Herman \cite{H}.
\end{remark}

We first verify that the neighborhood base at $\id \in \Diff_+^{1+\alpha}(\S1)$ is compatible with
the group structure. In other words, $\Diff_+^{1+\alpha}(\S1)$ is a partial topological group
in the sense of Gardiner and Sullivan \cite{GS}.

\begin{proposition}\label{partial}
The $C^{1+\alpha}$-modulus $p_{1+\alpha}$ satisfies the following:
\begin{enumerate}
\item
If $p_{1+\alpha}(g_n) \to 0$ and $p_{1+\alpha}(h_n) \to 0$ as $n \to \infty$ 
then $p_{1+\alpha}(g_n \circ h_n) \to 0$; 
\item
If $p_{1+\alpha}(g_n) \to 0$ as $n \to \infty$ then
$p_{1+\alpha}(g_n^{-1}) \to 0$.
\end{enumerate}
\end{proposition}

\begin{proof}
(1) It is obvious that $g_n \circ h_n \to \id$ and $(\widetilde{g_n \circ h_n})'(x) =\widetilde{g_n}'(\widetilde{h_n}(x))\widetilde{h_n}'(x)\to 1$ uniformly. 
Concerning the convergence of $c_\alpha$,
we have
\begin{align*}
&\quad \frac{|(\widetilde{g_n \circ h_n})'(x)-(\widetilde{g_n \circ h_n})'(y)|}{|x-y|^{\alpha}}\\
&\leq
\frac{|\widetilde{g_n}'(\widetilde{h_n}(x))\widetilde{h_n}'(x)-\widetilde{g_n}'(\widetilde{h_n}(y))\widetilde{h_n}'(x)|}{|x-y|^{\alpha}}
+\frac{|\widetilde{g_n}'(\widetilde{h_n}(y))\widetilde{h_n}'(x)-\widetilde{g_n}'(\widetilde{h_n}(y))\widetilde{h_n}'(y)|}{|x-y|^{\alpha}}\\
&\leq
\frac{c_{\alpha}(g_n)|\widetilde{h_n}(x)-\widetilde{h_n}(y)|^{\alpha}|\widetilde{h_n}'(x)|}{|x-y|^{\alpha}}
+|\widetilde{g_n}'(\widetilde{h_n}(y))|c_{\alpha}(h_n).
\end{align*}
As $c_\alpha(g_n), c_\alpha(h_n) \to 0$ and $\widetilde{g_n}'(x), \widetilde{h_n}'(x) \to 1$ uniformly, we see that
$c_\alpha(g_n \circ h_n) \to 0$ as $n \to \infty$.

(2) It is obvious that $g_n^{-1} \to \id$ and $(\widetilde{g_n^{-1}})'(x) =1/(\widetilde{g_n}'(\widetilde{g_n^{-1}}(x)) \to 1$ uniformly. 
Concerning the convergence of $c_\alpha$,
we have
\begin{align*}
\frac{|(\widetilde{g_n^{-1}})'(x)-(\widetilde{g_n^{-1}})'(y)|}{|x-y|^{\alpha}}
&=
\frac{|\widetilde{g_n}'(\widetilde{g_n^{-1}}(x))-\widetilde{g_n}'(\widetilde{g_n^{-1}}(y))|}{|x-y|^{\alpha}|\widetilde{g_n}'(\widetilde{g_n^{-1}}(x))||\widetilde{g_n}'(\widetilde{g_n^{-1}}(y))|}\\
&\leq
\frac{c_{\alpha}(g_n)|\widetilde{g_n^{-1}}(x)-\widetilde{g_n^{-1}}(y)|^{\alpha}}{|x-y|^{\alpha}|\widetilde{g_n}'(\widetilde{g_n^{-1}}(x))||\widetilde{g_n}'(\widetilde{g_n^{-1}}(y))|}.
\end{align*}
As $c_\alpha(g_n) \to 0$ and $\widetilde{g_n}'(x), (\widetilde{g_n^{-1}})'(x) \to 1$ uniformly, we see that
$c_\alpha(g_n^{-1}) \to 0$ as $n \to \infty$.
\end{proof}

In fact, we see more: $\Diff_+^{1+\alpha}(\S1)$ is a topological group.

\begin{proposition}\label{top}
With respect to the right uniform topology, $\Diff_+^{1+\alpha}(\S1)$ is a topological group.
\end{proposition}

\begin{proof}
According to Lemma 1.1 in \cite{GS}, we have only to show that the adjoint map is continuous at $\id$;
if $p_{1+\alpha}(g_n) \to 0$ as $n \to \infty$, then $p_{1+\alpha}(h \circ g_n \circ h^{-1}) \to 0$ for every $h \in \Diff_+^{1+\alpha}(\S1)$.
We have that $h \circ g_n \circ h^{-1} \to \id$ and 
$$
(\widetilde{h \circ g_n \circ h^{-1}})'(x)=\frac{\widetilde{h}'(\widetilde{g_n \circ h^{-1}}(x))}{\widetilde{h}'(\widetilde{h^{-1}}(x))} \widetilde{g_n}'(\widetilde{h^{-1}}(x)) \to 1
$$ 
uniformly. Furthermore,
\begin{align*}
&\quad \frac{|(\widetilde{h \circ g_n \circ h^{-1}})'(x)-(\widetilde{h \circ g_n \circ h^{-1}})'(y)|}{|x-y|^\alpha}\\
&=
\left|\frac{\widetilde{h}'(\widetilde{g_n \circ h^{-1}}(x))}{\widetilde{h}'(\widetilde{h^{-1}}(x))} \widetilde{g_n}'(\widetilde{h^{-1}}(x))-\frac{\widetilde{h}'(\widetilde{g_n \circ h^{-1}}(y))}{\widetilde{h}'(\widetilde{h^{-1}}(y))} \widetilde{g_n}'(\widetilde{h^{-1}}(y))\right|
\cdot |x-y|^{-\alpha},
\end{align*}
which is uniformly asymptotic to
$$
\frac{|\widetilde{g_n}'(\widetilde{h^{-1}}(x))-\widetilde{g_n}'(\widetilde{h^{-1}}(y))|}{|x-y|^\alpha} \leq \frac{c_\alpha(g_n)|\widetilde{h^{-1}}(x)-\widetilde{h^{-1}}(y)|^\alpha}{|x-y|^\alpha}.
$$
Because $c_\alpha(g_n) \to 0$, we see that $c_\alpha(h \circ g_n \circ h^{-1}) \to 0$ as $n \to \infty$.
\end{proof}

As every circle diffeomorphism is symmetric, $\Diff_+^{1+\alpha}(\S1)$ is a subgroup of $\Sym$.
We will characterize an element $g$ of $\Diff_+^{1+\alpha}(\S1)$ in terms of the
quasisymmetry quotient of $g$. 
This was shown in Lemma 5 in Carleson \cite{Car} 
(see also Section 9 of Gardiner and Sullivan \cite{GS}). 
The following statement and a detailed proof can be found in Theorem 7.1 of \cite{Mat0} and its corollary.

\begin{theorem}\label{qs-diff}
For a fixed $\alpha \in (0,1)$, we assume that there is some $b \geq 0$ such that
the lift $\widetilde g$ of $g \in \Sym$ satisfies
$$
(1+bt^\alpha)^{-1} \leq m_{\widetilde g}(x,t) \leq 1+bt^\alpha
$$
for every $x \in [0,1)$ and every $t \in (0,1/2]$.
Then, $g$ belongs to $\Diff_+^{1+\alpha}(\S1)$, and $c_\alpha(g)$ depends only on $b$ and
tends to $0$ uniformly as $b \to 0$.
Moreover, $\widetilde g'(x)$ is uniformly bounded from above and away from $0$
by constants depending only on $b$ with $\alpha$ fixed, which tend to $1$ as $b \to 0$.
\end{theorem}

Conversely, every element $g \in \Diff_+^{1+\alpha}(\S1)$ $(\alpha \in (0,1))$ belongs to $\Sym$
with a gauge function for symmetry of order $O(t^\alpha)$. More precisely, we have the following:

\begin{proposition}\label{diff-qs}
For $g \in \Diff_+^{1+\alpha}(\S1)$, there is a constant $b \geq 0$ such that
$$
(1+bt^\alpha)^{-1} \leq \ m_{\widetilde g}(x,t) \leq 1+bt^\alpha 
$$
for every $x \in [0,1)$ and every $t \in (0,1/2]$,
where $b$ can be taken depending only on $c=c_\alpha(g)$ when $c \leq 1$ and
tends to $0$ as $c \to 0$. 
\end{proposition}

For the proof, we need a simple claim.

\begin{proposition}\label{derivative}
Every $g \in \Diff_+^{1+\alpha}(\S1)$ satisfies
$$
1-c_\alpha(g)<1-c_\alpha(g)(1/2)^\alpha \leq \widetilde g'(x) \leq 1+c_\alpha(g)(1/2)^\alpha <1+c_\alpha(g).
$$
\end{proposition}

\begin{proof}
As $\int_0^{1} \widetilde g'(x)dx=1$,
there exists some $x_0 \in [0,1]$ such that $\widetilde g'(x_0) \geq 1$.
Likewise, there exists some $x'_0 \in [0,1]$ such that $\widetilde g'(x'_0) \leq 1$.
The H\"older continuity of $\widetilde g'$ implies that
$$
|\widetilde g'(x)-\widetilde g'(x_0)| \leq c_\alpha(g)|x-x_0|^\alpha \leq c_\alpha(g)(1/2)^\alpha
$$
for every $x \in \R$ with $|x-x_0| \leq 1/2$, and the same is true for $x'_0$. Then,
using the periodicity $\widetilde g'(x+1)=\widetilde g'(x)$, we have
the assertion.
\end{proof}

\noindent
{\it Proof of Proposition \ref{diff-qs}.}
The mean value theorem asserts that there are $\xi_+$ and $\xi_-$ such that
\begin{align*}
\widetilde g(x+t)-\widetilde g(x) &= t \widetilde g'(\xi_+) \quad (x< \xi_+ <x+t);\\
\widetilde g(x)-\widetilde g(x-t) &= t \widetilde g'(\xi_-) \quad (x-t< \xi_- <x).
\end{align*}
This gives
$$
m_{\widetilde g}(x,t)=1+\frac{\widetilde g'(\xi_+)-\widetilde g'(\xi_-)}{\widetilde g'(\xi_-)};\quad
m_{\widetilde g}(x,t)^{-1}=1+\frac{\widetilde g'(\xi_-)-\widetilde g'(\xi_+)}{\widetilde g'(\xi_+)}.
$$
Here, we can see that
$$
|\widetilde g'(\xi_+)-\widetilde g'(\xi_-)| \leq c_\alpha(g)|\xi_+-\xi_-|^\alpha \leq c_\alpha(g)(2t)^\alpha
$$ 
by the H\"older continuity of $\widetilde g'$. 
Proposition \ref{derivative} gives the lower estimate of $\widetilde g'$.
Moreover, as $g$ is a diffeomorphism,
there is some $c_0>0$ depending on $g$ such that 
$\widetilde g'(x) \geq c_0$. Therefore,
$$
m_{\widetilde g}(x,t)^{\pm 1} \leq 1+\frac{2^\alpha c_\alpha(g)}{\max\{1-c_\alpha(g)(1/2)^\alpha,c_0\}}\,t^\alpha.
$$
We set the coefficient of $t^\alpha$ as $b$.
If $c_\alpha(g) \leq 1$, then $1-c_\alpha(g)(1/2)^\alpha>0$ and $b$ depend only on 
$c=c_\alpha(g)$. Moreover, $b \to 0$ as $c \to 0$.
\qed
\medskip

Now we see that $g \in \Diff_+^{1+\alpha}(\S1)$ if and only if $m_{\widetilde g}(x,t)^{\pm 1}=1+O(t^\alpha)$ $(t \to 0)$.
Hereafter, we use a constant 
$$
b_\alpha(g)
=\sup_{0 \leq x < 1,\, 0< t \leq 1/2} \max_{\epsilon=\pm 1} \,
\frac{m_{\widetilde g}(x,t)^{\epsilon}-1}{t^\alpha}.
$$
Then, $m_{\widetilde g}(x,t)^{\pm 1} \leq 1+b_\alpha(g)t^\alpha$ and
the quasisymmetry constant satisfies $M(g) \leq 1+b_\alpha(g)$.
The convergence $b_\alpha(g) \to 0$ (and so on) is considered for a sequence of $g$, and 
this is said to be quantitative as $c_\alpha(g) \to 0$ (and so on)
if there is a majorant of $b_\alpha(g)$ in terms of $c_\alpha(g)$.

\begin{corollary}\label{equiv}
For $g \subset \Diff_+^{1+\alpha}(\S1)$, we have that
$c_\alpha(g) \to 0$ if and only if $b_\alpha(g) \to 0$ quantitatively.
Moreover,
under the extra assumption that $g$ is 
normalized so that it fixes the three points on $\S1$ $(g \in \QS_*)$,
$p_{1+\alpha}(g) \to 0$ if and only if $b_\alpha(g) \to 0$ or $c_\alpha(g) \to 0$ quantitatively.
\end{corollary}

\begin{proof}
The first statement directly follows from Theorem \ref{qs-diff} and Proposition \ref{diff-qs}.
For the second statement, we have only to show that $b_\alpha(g) \to 0$ or $c_\alpha(g) \to 0$ implies
$p_{1+\alpha}(g) \to 0$ quantitatively 
under the normalization. Theorem \ref{qs-diff} or Proposition \ref{derivative} verifies that $\widetilde g'$
converge to $1$ uniformly. Moreover, as $M(g) \leq 1+b_\alpha(g) \to 1$ and $g$ are normalized,
$g$ converge to $\id$ uniformly. Hence, we obtain $p_{1+\alpha}(g) \to 0$.
\end{proof}

Finally, in this section,
we prepare the investigation of $\Diff_+^{1+\alpha}(\S1)$ by the quasiconformal extension to $\D$. 
This will be completed in the next section.
We recall that,
as $\Diff_+^{1+\alpha}(\S1) \subset \Sym$, there is a quasiconformal extension that is
asymptotically conformal. We look at the decay order of its complex dilatation 
close to the boundary.

\begin{theorem}\label{diff-qc}
For every $g \in \Diff_+^{1+\alpha}(\S1)$, there exists a
quasiconformal extension $f \in \AC(\D)$ of $g$ whose complex dilatation $\mu$ belongs to $\Bel_0^\alpha(\D)$.
Here, $\Vert \mu \Vert_{\infty,\alpha}$ tends to $0$ quantitatively as 
$b_{\alpha}(g) \to 0$ or $c_{\alpha}(g) \to 0$.
\end{theorem}

\begin{proof}
By Proposition \ref{diff-qs}, the lift $\widetilde g:\R \to \R$ of $g$ satisfies
$m_{\widetilde g}(x,t)^{\pm 1} \leq 1+b_\alpha(g)t^\alpha$ for a finite constant $b_\alpha(g) \geq 0$.
Then, by Theorem \ref{Carleson1}, 
the complex dilatation $\mu_F(z)$ of the Beurling--Ahlfors extension $F(z)$ of $\widetilde g$
satisfies
$|\mu_F(z)| \leq 4b_\alpha(g) y^\alpha$
for every $z=x+iy \in \H2$.
The projection $f:\D-\{0\} \to \D-\{0\}$ of $F$ under the holomorphic universal cover
$u:\H2 \to \D-\{0\}$ $(z \mapsto \zeta=e^{2\pi iz})$
is defined as $e_{\rm BA}(g)$ after filling $0$.

The complex dilatation $\mu$ of $f=e_{\rm BA}(g)$ satisfies
$$
|\mu(\zeta)|=|\mu_F(z)|=|\mu_F((\log \zeta)/(2\pi i))|
$$
for every $\zeta \in \D$.
As ${\rm Im}\,[(\log \zeta)/(2\pi i)]=-\log|\zeta|/(2\pi)$, the condition $|\mu_F(z)| \leq 4b_\alpha(g) y^\alpha$ yields 
$$
|\mu(\zeta)| \leq \frac{4b_\alpha(g)}{(2\pi)^\alpha}(-\log|\zeta|)^\alpha. 
$$
As $-\log |\zeta|$ is comparable to $1-|\zeta|$ near $|\zeta|=1$, 
we can find a continuous increasing function $d:[0,1) \to [1,\infty)$ with $\lim_{t \to 0} d(t)=1$ such that 
$$
|\mu(\zeta)| \leq \frac{4b_\alpha(g)}{(2\pi)^\alpha} d(\Vert \mu \Vert_\infty)(1-|\zeta|)^\alpha 
$$
for every $\zeta \in \D$.
Moreover, if $c_{\alpha}(g) \to 0$, then
$b_\alpha(g) \to 0$ by Proposition \ref{diff-qs}, and hence, $M(g) \to 1$ which implies 
that $\Vert \mu \Vert_\infty \to 0$ (see Theorem I.5.2 in \cite{Leh}). 
Therefore, we see that $\Vert \mu \Vert_{\infty,\alpha} \to 0$ quantitatively as 
$b_{\alpha}(g) \to 0$ or $c_{\alpha}(g) \to 0$.
\end{proof}

\section{Quasiconformal characterization of circle diffeomorphisms}\label{6}
We will establish the relationships among the following three indices
quantitatively: the exponent of H\"older continuity of the derivative of
a circle diffeomorphism $g$; the decay order of the complex dilatation of quasiconformal extension of $g$; and
the decay order of the Schwarzian derivative of the corresponding conformal homeomorphism.
We have seen the equivalence of the last two quantities (Theorem \ref{equivalence1}) and 
the implication of the second one from the first (Theorem \ref{diff-qc}).

The new addition is the converse of the statement of Theorem \ref{diff-qc}.
In Theorem \ref{ac-sym} and Corollary \ref{acsym}, we have seen that an asymptotically conformal homeomorphism $f \in \AC(\D)$
extends to a symmetric homeomorphism $g \in \Sym$ and provided a certain estimate of the gauge function
for symmetry in terms of the decay order of $\mu_f$. 
The order of the gauge function and the H\"older continuity of the derivative are related to each other as shown
in Theorem \ref{qs-diff} and Proposition \ref{diff-qs}.

However, the order of the gauge function is reduced to $\alpha/2$ from the decay order $\alpha$ of $\mu_f$
according to Theorem \ref{ac-sym}. Moreover, in the course of transforming the situation from $\H2$ to $\D$,
we need a certain normalization on $g \in \Sym$ to obtain a quantitative estimate.
A summary of these situations is the following:

\begin{lemma}\label{sofar}
For a $K$-quasiconformal self-homeomorphism $f$ of $\D$ with 
complex dilatation $\mu \in \Bel^\alpha_0(\D)$,
its boundary extension $g$ belongs to $\Diff_+^{1+\alpha/2}(\S1)$. 
In addition, under the normalization such as $f(0)=0$ or $g \in \QS_*$, the derivative of $g$ is
uniformly bounded from above and away from $0$. More precisely,
there is a constant
$D=D(\alpha,K,\ell) \geq 1$ depending only on
$\alpha$, $K$ and $\ell$ with $\Vert \mu_f \Vert_{\infty,\alpha} \leq \ell$ such that
$$
\frac{1}{D} \leq \widetilde g'(x) \leq D
$$
for every $x \in \R$.
\end{lemma}

\begin{proof}
We assume that $f$ fixes $0$. In this case, $f$ lifts to 
the quasiconformal self-homeo\-morphism $F$ of $\H2$ under the holomorphic universal cover $u:\H2 \to \D-\{0\}$
$(z \mapsto \zeta=e^{2\pi iz})$. 
The complex dilatation of $F$ satisfies 
$$
|\mu_{F}(z)|=|\mu(\zeta)| \leq (1-|\zeta|)^\alpha \ell \leq (2\pi)^\alpha \ell y^\alpha \quad (z=x+iy \in \H2).
$$
Then, Theorem \ref{ac-sym} is applied for 
$\widetilde \varepsilon(y)=(2\pi)^\alpha \ell y^{\alpha}$ to verify that the quasisymmetry quotient of $\widetilde g:\R \to \R$,
which is the boundary extension of $F$ as well as the lift of $g$, satisfies 
$$
m_{\widetilde g}(x,t)^{\pm 1} \leq 1+c \widetilde \varepsilon(t^{1/2}) +R t^{1/2} \leq 1+b t^{\alpha/2}
$$
for every $x \in [0,1)$ and every $t \in (0,1/2]$,
where $b=b(K,\ell)>0$ is a constant depending only on $K$ and $\ell$. 
Then, Theorem \ref{qs-diff} asserts that $g$ belongs to $\Diff_+^{1+\alpha/2}(\S1)$. 
Moreover, the derivative $\widetilde g'(x)$ is estimated in terms of $\alpha$ and $b$ by the same theorem.

For a general $f$ not necessarily fixing $0$, we take $\phi \in \Mob(\D)$ such that $\phi \circ f(0)=0$.
The complex dilatation of $\phi \circ f$ is the same as that of $f$. Then, we can apply the previous argument
to $\phi \circ f$; we obtain $\phi \circ g \in \Diff_+^{1+\alpha/2}(\S1)$, where the same symbol $\phi \in \Mob(\S1)$ denotes the
boundary extension of $\phi \in \Mob(\D)$. This in particular shows that $g$ itself
belongs to $\Diff_+^{1+\alpha/2}(\S1)$. Moreover, if $g$ is normalized, Proposition \ref{Teich} below
shows that $|f(0)| \leq r$ for some $r=r(K) \in [0,1)$. Then, $\phi$ satisfies
$$
\frac{1-r}{1+r} \leq |\phi'(z)| \leq \frac{1+r}{1-r} \quad (z \in \overline{\D}). 
$$
From the uniform boundedness of $(\widetilde{\phi \circ g})'(x)$ by the previous argument, we also see that
$\widetilde g'(x)$ is uniformly bounded from above and away from $0$.
\end{proof}

We often 
compare the condition $f(0)=0$ with our normalization fixing $1$, $i$, and $-1$ for $f=f^\mu \in \QC(\D)$.
The following proposition ensures that their differences are 
small.

\begin{proposition}\label{Teich}
There is a constant $r=r(K) \in [0,1)$ depending only on $K$ such that
every $K$-quasiconformal homeomorphism $f \in \QC(\D)$ fixing $1$, $i$, and $-1$ 
satisfies $|f(0)| \leq r$.
\end{proposition}

\begin{proof}
We assume that $f \in \QC(\D)$ extends to the quasiconformal self-homeo\-morphism of $\Chat$ by reflection with respect to $\S1$. 
The distortion theorem for cross ratio due to Teich\-m\"ul\-ler (see Section III.D of \cite{Ah0} and \cite{Kra}) implies that
for any four distinct points $z_1,z_2,z_3,z_4 \in \Chat$, the hyperbolic distance between
the cross ratios $[z_1,z_2,z_3,z_4]$ and $[f(z_1),f(z_2),f(z_3),f(z_4)]$
in $\C-\{0,1\}$ is bounded by $\log K$.
We choose $z_1=0$ and $z_2=\infty$. If we choose two distinct points from $\{1,i,-1\}$
for $z_3$ and $z_4$, we see that $f(0)=f(\infty)^*$ cannot be close to $\S1$ except in some neighborhoods of
$z_3$ and $z_4$ within a distance depending only on $K$. 
By considering all such choices from $\{1,i,-1\}$, we obtain the assertion.
\end{proof}

The full converse of Theorem \ref{diff-qc} should be a statement that if the complex dilatation $\mu_f$ of
$f \in \AC(\D)$ is in $\Bel^\alpha_0(\D)$, then the boundary extension $g$ of $f$ 
belongs to $\Diff_+^{1+\alpha}(\S1)$ for the same $\alpha$. We will prove this, which is the improvement of
the weaker consequence $g \in \Diff_+^{1+\alpha/2}(\S1)$ in Lemma \ref{sofar}.
We also do this quantitatively.
The claim on the derivative of $g$ in this lemma is still necessary for the estimation of 
the $C^{1+\alpha}$-modulus $p_{1+\alpha}(g)$ as well as for Theorem \ref{mori} below.

We need distortion estimates of quasiconformal self-homeomorphisms of $\D$,
which are variants of the Mori theorem. The first one is its direct consequence.

\begin{proposition}\label{mori0}
Let $f$ be a $K$-quasiconformal self-homeomorphism of $\D$ with $f(0)=0$.
Then, 
$$
\frac{1}{16^K}(1-|z|)^K \leq 1-|f(z)| \leq 16(1-|z|)^{1/K}
$$
is satisfied for every $z \in \D$. 
\end{proposition}

\begin{proof}
The Mori theorem (see Section III.C of \cite{Ah0} and Theorem II.3.2 of \cite{LV}) assert that
$$
|f(w)-f(z)| \leq 16|w-z|^{1/K}
$$
for any $w$ and $z$ in $\overline{\D}$. We choose $w=z/|z| \in \S1$ for every $z \in \D$.
Then, the upper inequality follows from $1-|f(z)| \leq |f(w)-f(z)|$.
Considering $f^{-1}$, we obtain the other inequality.
\end{proof}

We can remove the powers $1/K$ and $K$ in the inequalities of Proposition \ref{mori0} 
if the complex dilatation belongs to our class $\Bel_0^{\alpha}(\D)$.
The following result verifies this, which will be crucial in our arguments.

\begin{theorem}\label{mori}
Let $f^\mu$ be a normalized $K$-quasiconformal self-homeomorphism of $\D$ with 
$\mu \in \Bel_0^\alpha(\D)$ and $\Vert \mu \Vert_{\infty,\alpha} \leq \ell$. 
Then, there is a constant $A=A(\alpha,K,\ell)\geq 1$ depending only on 
$\alpha$, $K$, and $\ell$ such that
$$
\frac{1}{A}(1-|z|) \leq 1-|f^\mu(z)| \leq A(1-|z|)
$$
for every $z \in \D$.
\end{theorem}

\begin{proof}
For the moment, 
we prove the inequalities for $f \in \AC(\D)$ with $f(0)=0$,
whose complex dilatation $\mu$ satisfies the same assumption as in the statement. 
Let 
$$
t_0=\min\{(2\ell)^{-2/\alpha},1/4\}>0.
$$
It is easy to show 
the inequalities for $z \in \D$ with $1-|z| \geq t_0$. 
Indeed, using Proposition \ref{mori0}, we have
\begin{align*}
\frac{t_0^K}{16^K}(1-|z|) \leq
\frac{t_0^K}{16^K} \leq \frac{1}{16^K}(1-|z|)^K
\leq 1-|f(z)| \leq 1 \leq t_0^{-1}(1-|z|).
\end{align*}
Thus, we may assume that $1-|z| < t_0$ hereafter.

Let $t=1-|z| <t_0$ for a given point $z \in \D$.
We define a Beltrami coefficient $\mu_t(\zeta)$ by
setting $\mu_t(\zeta) = \mu(\zeta)$ on $\{\zeta \in \D \mid |\zeta| \leq 1-\sqrt{t}\}$ and $\mu_t(\zeta) \equiv 0$ elsewhere.
Let $f_t$ be the quasiconformal self-homeomorphism of $\D$ with the complex dilatation $\mu_t$ and with $f_t(0)=0$.
Let $h_t$ be the quasiconformal self-homeomorphism of $\D$ such that $f = h_t \circ f_t$. 
As
$|\mu(\zeta)| \leq \ell(1-|\zeta|)^\alpha$, 
we see that $|\mu_{h_t}(w)| \leq \ell t^{\alpha/2} < 1/2$ for $w \in \D$, 
which implies that the maximal dilatation $K_t$ of $h_t$ satisfies
$$
\frac{1}{K_t} \geq \frac{1-\ell t^{\alpha/2}}{1+\ell t^{\alpha/2}} \geq 1-2\ell t^{\alpha/2}; \quad
K_t \leq \frac{1+\ell t^{\alpha/2}}{1-\ell t^{\alpha/2}} \leq 1+4 \ell t^{\alpha/2}<3.
$$

First, we apply a distortion theorem to the conformal homeomorphism $f_t(\zeta)$ restricted to 
$|\zeta| >1-\sqrt{t}$.
In fact, we may assume that $f_t$ is a conformal homeomorphism of an annulus $\{1-\sqrt{t}<|\zeta|<1/(1-\sqrt{t})\}$
by the reflection principle. Moreover, $f_t$ is also an $K$-quasiconformal self-homeomorphism of $\D$
whose complex dilatation satisfies $\Vert \mu_{f_t}\Vert_{\infty,\alpha} \leq \ell$ independently of $t$.
Then, we see from Lemma \ref{sofar} that 
there is a constant $D=D(\alpha,K,\ell) \geq 1$
independent of $t$ such that the derivative $f_t'$ satisfies
$D^{-1} \leq |f_t'(\xi)| \leq D$ for every $\xi \in \S1$.

The Koebe distortion theorem (Proposition \ref{Koebe}) in the disk $\Delta(\xi,\sqrt{t})$ of radius $\sqrt{t}$ and center $\xi=z/|z|$ 
yields an upper estimate
$$
1-|f_t(z)| \leq |f_t(z)-f_t(\xi)| \leq
(|f'_t(\xi)|\sqrt{t})\frac{t/\sqrt{t}}{(1-t/\sqrt{t})^2} \leq 4Dt 
$$
if $t<1/4$. A lower estimate is more complicated.
Proposition \ref{Koebe} shows that
$$
|f_t'(z)| \geq |f_t'(\xi)|\frac{1-t/\sqrt{t}}{(1+t/\sqrt{t})^3} \geq \frac{4}{27D}
$$
with $t<1/4$.
We consider the reflection $z^*$ of $z$ with respect to $\S1$. The Koebe distortion theorem
applied after sending $z$ to $\xi$ by a conformal self-homeomorphism of the disk $\Delta(\xi,\sqrt{t})$
(see Corollary 1.5 in \cite{P}) gives 
$$
|f_t(z)-f_t(z^*)| \geq (1-(t/\sqrt{t})^2)(|f_t'(z)|\sqrt{t})\cdot \frac{2(t/\sqrt{t})}{4(1-(t/\sqrt{t})^2)}
\geq \frac{2t}{27D}.
$$
As $f_t(z^*)$ is the reflection of $f_t(z)$ with respect to $\S1$, 
$1-|f_t(z)|$ is nearly a half of $|f_t(z)-f_t(z^*)|$ if it is small; for example,
$1-|f_t(z)| \geq 9|f_t(z)-f_t(z^*)|/20$ if $1-|f_t(z)| \leq 2/11$.
This in particular shows that
$$
1-|f_t(z)| \geq \frac{9}{20}\cdot \frac{2t}{27D} =\frac{t}{30D}\,\left(\leq \frac{2}{11}\right). 
$$

Next, we apply Proposition \ref{mori0} to the quasiconformal self-homeomorphism $h_t$ of $\D$.
It implies that
\begin{align*}
1-|h_t(w)| &\leq 16(1-|w|)^{1/K_t} \leq 16(1-|w|)^{1-2\ell t^{\alpha/2}};\\
1-|h_t(w)| &\geq \frac{1}{16^{K_t}}(1-|w|)^{K_t} \geq \frac{1}{16^3}(1-|w|)^{1+4\ell t^{\alpha/2}}
\end{align*}
for every $w \in \D$. Then, by setting $w=f_t(z)$, we have
$$
\frac{1}{16^3}(t/(30D))^{1+4\ell t^{\alpha/2}} 
\leq 1-|f(z)| \leq 16(4Dt)^{1-2\ell t^{\alpha/2}}.
$$
Dividing these inequalities by $t=1-|z|$ and taking the logarithm, we obtain 
\begin{align*}
-3\log(50D)+4\ell t^{\alpha/2} \log(t/(30D))
&\leq \log \frac{1-|f(z)|}{1-|z|}\\
&\leq \log(64D)-2\ell t^{\alpha/2} \log(4Dt).
\end{align*}
This shows that the middle term is bounded from above and below independently of $t$, and hence
$(1-|f(z)|)/(1-|z|)$ is also bounded from above and away from $0$.
Thus, we can find a constant $A'=A'(\alpha,K,\ell) \geq 1$ such that
$$
\frac{1}{A'}(1-|z|) \leq 1-|f(z)| \leq A'(1-|z|)
$$
for the case of $1-|z| < t_0$ as well as for the previous case $1-|z| \geq t_0$.

Now we consider the normalized quasiconformal homeomorphism $f^\mu \in \QC(\D)$.
Proposition \ref{Teich} asserts that there is $r=r(K) \in [0,1)$ such that
$|f^\mu(0)| \leq r$. We take a M\"obius transformation $\phi \in \Mob(\D)$ such that $\phi \circ f^\mu(0)=0$.
Then, $f=\phi \circ f^\mu$ satisfies the above inequalities. 
Moreover, $|f^\mu(0)| \leq r$
implies that 
$$
\frac{1-r}{1+r} \leq |\phi'(z)| \leq \frac{1+r}{1-r} \quad (z \in \D). 
$$
Because
$$
(\min_{z \in \D}|\phi'(z)|)(1-|f(z)|) \leq 1-|f^\mu(z)| \leq (\max_{z \in \D}|\phi'(z)|)(1-|f(z)|), 
$$
we can choose $A=A'(1+r)/(1-r)$ for the required inequalities, which depends only on $\alpha$, $K$, and $\ell$.
\end{proof}

This theorem has several consequences.

\begin{proposition}\label{group}
For any $\mu$ and $\nu$ in $\Bel_0^\alpha(\D)$, the composition $\mu \ast \nu^{-1}$ also belongs to $\Bel_0^\alpha(\D)$.
Hence, $\Bel_0^\alpha(\D)$ is a subgroup of $\Bel(\D)$. 
\end{proposition}

\begin{proof}
We apply Theorem \ref{mori} to $\zeta=f^{\nu}(z)$ in the formula
$$
\mu \ast \nu^{-1}(\zeta)=\frac{\mu(z)-\nu(z)}{1-\overline{\nu(z)}\mu(z)}\cdot\frac{\partial f^{\nu}(z)}{\overline{\partial f^{\nu}(z)}}.
$$
Then, $\rho_{\D}^{\alpha}(\zeta) \leq (2A)^\alpha \rho_{\D}^{\alpha}(z)$, from which we have
$$
\Vert \mu \ast \nu^{-1} \Vert_{\infty,\alpha} \leq 
\frac{(2A)^\alpha}{1-\Vert \mu \Vert_\infty \Vert \nu \Vert_\infty} \Vert \mu-\nu \Vert_{\infty,\alpha}.
$$
The statement follows from this inequality.
\end{proof}

\begin{corollary}\label{inverse}
If $\nu \in  \Bel_0^\alpha(\D)$ then $\nu^{-1} \in \Bel_0^\alpha(\D)$. More precisely, every
$\nu \in  \Bel_0^\alpha(\D)$ with $\Vert \nu \Vert_{\infty,\alpha} \leq \ell$ 
and $\Vert \nu \Vert_{\infty} \leq k<1$ 
satisfies
$\Vert \nu^{-1} \Vert _{\infty,\alpha} \leq \widetilde A \Vert \nu \Vert_{\infty,\alpha}$
for a constant $\widetilde A=\widetilde A(\alpha,k,\ell) \geq 1$.
\end{corollary}

\begin{proof}
As a special case of the above inequality by setting $\mu=0$, we have
$$
\Vert \nu^{-1} \Vert_{\infty,\alpha} \leq (2A)^\alpha \Vert \nu \Vert_{\infty,\alpha}.
$$
Then, setting $\widetilde A=(2A)^\alpha$ gives the statement, as $A$ depends only on $\alpha$, $k$, and $\ell$ 
by Theorem \ref{mori}.
\end{proof}

Now we explain the converse of Theorem \ref{diff-qc} as well as other 
equivalent conditions for $g \in \QS$ to belong to $\Diff_+^{1+\alpha}(\S1)$.
We supply the following notation.

\begin{definition}
For a bounded holomorphic quadratic differential $\varphi=\varphi(z)dz^2 \in B(\D^*)$,
we define a new norm by
$$
\Vert \varphi \Vert_{\infty,\alpha}=\sup_{z \in \D^*} \rho^{-2+\alpha}_{\D^*}(z)|\varphi(z)|.
$$
The Banach space of holomorphic quadratic differentials with this norm finite is given by
$$
B_0^{\alpha}(\D^*)=\{\varphi \in B(\D^*) \mid \Vert \varphi \Vert_{\infty,\alpha}<\infty \} \subset B_0(\D^*).
$$
\end{definition}

\begin{theorem}\label{main}
Let $\alpha$ be a constant with $0<\alpha<1$.
For a quasisymmetric homeomorphism $g \in \QS$, the following conditions are equivalent:
\begin{enumerate}
\item
$g$ belongs to $\Diff_+^{1+\alpha}(\S1)$;
\item
there is $\mu \in \Bel_0^\alpha(\D)$ such that $\pi(\mu)=[g] \in T$;
\item
$\beta([g]) \in \beta(T)$ is in $B_0^\alpha(\D^*)$.
\end{enumerate}
\end{theorem}

\begin{proof}
The implication $(1) \Rightarrow (2)$ is a reformulation of Theorem \ref{diff-qc}.
This was essentially proved by Carleson \cite{Car}. 
The equivalence $(2) \Leftrightarrow (3)$ has been reviewed in Theorem \ref{equivalence1},
where previous contributions to this equivalence are also mentioned.
We note that $(1) \Rightarrow (3)$ was also proved in
Tam and Wan \cite{TW} by using the harmonic extension of diffeomorphisms of $\S1$.
On the contrary, the converse $(2) \Rightarrow (1)$ was given in
Dyn$'$kin \cite{Dy}
based on his results on the pseudo\-analytic extension of
differentiable functions and independently in
Anderson, Cant\'on, and Fern\'andez \cite{ACF}, who relied on
a certain approximation theorem of quasiconformal maps on the disk by polynomials.
Theorem \ref{key} below proves $(2) \Rightarrow (1)$ in complex analytic methods and 
provides necessary results for our theorems on the Teich\-m\"ul\-ler space.
\end{proof}

For  
later purposes, we prepare the proposition that follows next.
We will use it for both $\mu \in \Bel^\alpha_0(\D)$ and its reflection $\mu^*$.
According to the different assumptions that we will impose on them, we address both cases separately. 

\begin{proposition}\label{distortion}
$(1)$ Let $f$ be a conformal homeomorphism of $\D^*$ 
with $f(\infty)=\infty$ and 
$\lim_{z \to \infty}f'(z)=1$
whose quasiconformal extension to $\D$
has the complex dilatation $\mu$
in $\Bel_0^\alpha(\D)$ with $\Vert \mu \Vert_{\infty,\alpha} \leq \ell$.
Then, there is a constant $B=B(\alpha,\ell) \geq 1$  
such that
$$
\frac{1}{B} \leq |f'(z)| \leq B
$$
for every $z \in \D^*$. $(2)$ Let $f$ be a conformal homeomorphism of $\D$ 
with $e^{-s} \leq |f'(0)| \leq e^s$
whose quasiconformal extension to $\D^*$
has the complex dilatation $\mu^*$
for $\mu \in \Bel_0^\alpha(\D)$ with $\Vert \mu \Vert_{\infty,\alpha} \leq \ell$.
Then, there is a constant $B'=B'(\alpha,\ell,s) \geq 1$ 
such that
$$
\frac{1}{B'} \leq |f'(z)| \leq B'
$$
for every $z \in \D$. 
\end{proposition}

\begin{proof}
(1) By Theorem \ref{basic1}, there is a constant $L=L(\alpha,\ell) \geq 0$ such that
$\beta_{\mu}(t) \leq 2Lt^\alpha/(t+2)$.
Because 
$$
\left|\frac{f''(z)}{f'(z)}\right| \leq \frac{\beta_{\mu}(t)}{t}
$$ 
for $t=|z|-1$, the integration along the radial segment connecting $(1+t)\xi$ and $\xi$
for any $\xi \in \S1$ gives
$$
\int^{(1+t)\xi}_{\xi} \left|\frac{d}{dz} \log f'(z)\right ||dz|
\leq L \int_{0}^{t}\frac{2t^{\alpha-1}}{t+2}dt.
$$ 
The right side term is bounded by $Lt^\alpha/\alpha$, which
implies that $\log f'$ extends continuously to $\S1$ (see Theorem 4.1 in Pommerenke and Warschawski \cite{PW}).
Moreover, by taking the limit as $t \to \infty$, we obtain
$$
|\log f'(\xi)| \leq \frac{L}{\alpha}+2L\int_1^\infty t^{\alpha-2}dt
=\frac{L}{\alpha}+\frac{2L}{1-\alpha}
$$
for every $\xi \in \S1$.
Then, the maximal principle yields that $|\log f'(z)| \leq 2L/(\alpha(1-\alpha))$ for every $z \in \D^*$. Hence,
by taking $B=\exp(2L/(\alpha(1-\alpha)))$, we obtain the assertion.

(2) By Corollary \ref{interior}, there is a constant $L'=L'(\alpha,\ell) \geq 0$ such that
$\bar \beta_{\mu^*}(t) \leq L't^\alpha$.
Because 
$$
\left|\frac{f''(z)}{f'(z)}\right| \leq \frac{\bar \beta_{\mu^*}(t)}{t}
$$ 
for $t=1-|z|$, the integration along the radial segment connecting $(1-t)\xi$ and $\xi$
for any $\xi \in \S1$ gives
$$
\int^{\xi}_{(1-t)\xi} \left|\frac{d}{dz} \log f'(z)\right ||dz|
\leq L' \int_{0}^{t}t^{\alpha-1}dt
=\frac{L't^\alpha}{\alpha}.
$$
Similarly to the above,
$\log f'$ extends continuously to $\S1$. By taking $t=1$, we obtain
$$
|\log f'(\xi)-\log f'(0)| \leq \frac{L'}{\alpha}
$$
for every $\xi \in \S1$. 
Then, the maximal principle yields that $|\log f'(z)-\log f'(0)| \leq L'/\alpha$ for every $z \in \D$. 
As $-s \leq \log |f'(0)| \leq s$, we have that $|\log |f'(z)|| \leq L'/\alpha+s$; hence
by taking $B'=\exp(L'/\alpha+s)$, we obtain the assertion.
\end{proof}

\begin{theorem}\label{key}
If $\mu \in \Bel_0^\alpha(\D)$, then $g \in \QS$ with $\pi(\mu)=[g]$
belongs to $\Diff_+^{1+\alpha}(\S1)$.
Moreover, if $g$ is normalized $(g \in \QS_*)$, then $p_{1+\alpha}(g) \to 0$ quantitatively as
$\Vert \mu \Vert_{\infty,\alpha} \to 0$.
\end{theorem}

\begin{proof}
We may assume that the normalized quasiconformal self-homeomorphism $f^\mu$ of $\D$
with the complex dilatation $\mu$ extends to $g \in \QS_*$.
We represent this $g$ by conformal welding.
The quasiconformal homeomorphism of $\Chat$
extended by the reflection of $f^\mu$ with respect to $\S1$ is also denoted by $f^\mu$.
Let $f_\mu$ be the normalized quasiconformal self-homeomorphism of $\Chat$
whose complex dilatation is $\mu$ on $\D$ and $0$ on $\D^*$, which satisfies 
$f_\mu(\infty)=\infty$ and
$\lim_{z \to \infty}f'_\mu(z)=1$.
We define the quasiconformal self-homeomorphism $f_\mu \circ (f^\mu)^{-1}$ of $\Chat$ by $f$,
which is conformal on $\D$ with $f(\D)=f_\mu(\D)$
and whose complex dilatation on $\D^*$ is $(\mu^*)^{-1}$, the inverse of the reflection of $\mu$.
Then, $g=f^{-1} \circ f_\mu$ on $\S1$.
We note that $(\mu^*)^{-1}=(\mu^{-1})^*$, where $\mu^{-1}$ belongs to $\Bel_0^\alpha(\D)$ 
and $\Vert \mu^{-1} \Vert_{\infty,\alpha}$ can be estimated in terms of $\Vert \mu \Vert_{\infty,\alpha}$ by Corollary \ref{inverse}.

We will estimate the modulus of continuity of 
the derivative of $g:\S1 \to \S1$ at $e^{2\pi ix} \in \S1$
in terms of $\beta_\mu$ and $\bar \beta_{(\mu^*)^{-1}}$. 
This is based on an argument given by
Anderson, Becker, and Lesley \cite{ABL}.
By Theorem \ref{basic1}, we see that $\beta_\mu(t) \leq L t^\alpha$ for some constant $L \geq 0$ 
tending to $0$ uniformly as
$\Vert \mu \Vert_{\infty,\alpha} \to 0$. By Corollary \ref{interior}, we also have that
$\bar \beta_{(\mu^*)^{-1}}(t) \leq L' t^\alpha$ for some constant $L' \geq 0$ with the same property as $L$;
if $\Vert \mu \Vert_{\infty,\alpha} \to 0$, then $\Vert \mu^{-1} \Vert_{\infty,\alpha} \to 0$, and hence
$L' \to 0$ uniformly.

Now we consider
the derivative of the lift $\widetilde g: \R \to \R$ at $x \in \R$ represented by
$$
\widetilde g'(x)=\lim_{s \to 0} \left| \frac{g(e^{2\pi i(x+s)})-g(e^{2\pi ix})}{e^{2\pi i(x+s)}-e^{2\pi ix}} \right|
=|g'(e^{2\pi ix})|,
$$
where $g'(e^{2\pi ix})$ is the directional derivative along the tangent of $\S1$ at $e^{2\pi ix}$.
We see that $g$ is continuously differentiable and
$$
g'(e^{2\pi ix})=(f_\mu)'(e^{2\pi ix})/f'(g(e^{2\pi ix})).
$$
Indeed, as in the proof of Proposition \ref{distortion},
if $\Vert \mu \Vert_{\infty,\alpha}<\infty$,
then $(f_\mu)'(z)$ $(z \in \D^*)$ has a non-vanishing
continuous extension to $\S1=\partial \D^*$.
This is also true for $f'(z)$ $(z \in \D)$.
As $g$ is normalized, Lemma \ref{sofar} 
asserts that $\widetilde g'(x) \leq D$ for a 
constant $D \geq 1$ uniformly bounded when $\Vert \mu \Vert_{\infty,\alpha} \to 0$.

The modulus of continuity of $\widetilde g'$ is defined by
$$
I(t;\widetilde g')=\sup_{|x-y| \leq t} |\widetilde g'(x)-\widetilde g'(y)|
$$
for every $t \in (0,1/2]$. We note that 
$$
c_\alpha(g)=\sup_{0<t \leq 1/2} \frac{I(t;\widetilde g')}{t^\alpha}.
$$
According to the mean value theorem, 
$|\widetilde g(x)-\widetilde g(y)| \leq D|x-y|$, 
and if $\widetilde g'(x) \geq \widetilde g'(y)>0$, then
$$
|\widetilde g'(x)-\widetilde g'(y)| \leq D \left|1-\frac{\widetilde g'(y)}{\widetilde g'(x)}\right| \leq D \left|\log \frac{\widetilde g'(y)}{\widetilde g'(x)}\right|.
$$
This yields $I(t;\widetilde g') \leq D I(t; \log \widetilde g')$. The case where $\widetilde g'(y) \geq \widetilde g'(x)>0$ deduces
the same estimate. Moreover,
\begin{align*}
I(t;\log \widetilde g')
&\leq  I(t;\log |f_\mu'(e^{2\pi i \, \bullet})|)+I(t;\log |f'(g(e^{2\pi i \, \bullet}))|)\\
&\leq I(t;\log |f_\mu'(e^{2\pi i \, \bullet})|)+I(Dt;\log |f'(e^{2\pi i \, \bullet})|).
\end{align*}

Here, we note that
$(\log f_\mu')'(z)=T_{f_\mu|_{\D^*}}(z)$ 
and
$(\log f')'(z)=T_{f|_{\D}}(z)$.
Taking a path of integration
including the circular arc $\gamma$ joining $e^{2\pi ix}(1+t)$ and $e^{2\pi iy}(1+t)$ in $\D^*$
for $e^{2\pi ix}, e^{2\pi iy} \in \S1$ with $|x-y| \leq t$, we obtain 
\begin{align*}
&\quad|\log |f_\mu'(e^{2\pi ix})|-\log |f_\mu'(e^{2\pi iy})|| 
\leq |\log f_\mu'(e^{2\pi ix})-\log f_\mu'(e^{2\pi iy})| \\
&\leq
\int_{e^{2\pi ix}}^{e^{2\pi ix}(1+t)} |T_{f_\mu|_{\D^*}}(z)||dz|
+\int_\gamma |T_{f_\mu|_{\D^*}}(z)||dz|
+\int_{e^{2\pi iy}(1+t)}^{e^{2\pi iy}} |T_{f_\mu|_{\D^*}}(z)||dz|\\
&\leq
2\int_0^{t} \frac{\beta_\mu(t)}{t} dt
+2\pi(1+t)\beta_\mu(t).
\end{align*}
This implies that, if $\beta_\mu(t) \leq L t^\alpha$, then 
$$
I(t;\log |f_\mu'(e^{2\pi i \, \bullet})|) \leq (2/\alpha+3\pi)L t^\alpha
$$
for every $t \in (0,1/2]$.
The same holds for $f'$; thus, 
$$
I(Dt;\log |f'(e^{2\pi i \, \bullet})|) \leq (2/\alpha+3\pi)L' D^\alpha t^\alpha.
$$
Hence, $I(t;\widetilde g')=O(t^\alpha)$, which means that $g$ belongs to $\Diff_+^{1+\alpha}(\S1)$
by definition. 

Under the normalization $g \in \QS_*$, we have seen that $D$ is uniformly bounded
as $\Vert \mu \Vert_{\infty,\alpha} \to 0$.
Because $L,\, L' \to 0$ as $\Vert \mu \Vert_{\infty,\alpha} \to 0$, 
this shows that $I(t;\widetilde g')/t^\alpha$ tends to $0$ uniformly, which means that $c_\alpha(g) \to 0$.
Then, by Corollary \ref{equiv}, this implies that $p_{1+\alpha}(g) \to 0$.
Thus, $p_{1+\alpha}(g) \to 0$ as
$\Vert \mu \Vert_{\infty,\alpha} \to 0$. All these sequences converge quantitatively.
\end{proof}

Condition (2) of Theorem \ref{main} says that there exists some Beltrami coefficient $\mu \in \Bel^{\alpha}_0(\D)$
whose Teich\-m\"ul\-ler projection $\pi(\mu)$ coincides with $[g]$ for a given $g \in \Diff_+^{1+\alpha}(\S1)$.
Alternatively, this means that $g \in \Diff_+^{1+\alpha}(\S1)$ has some quasiconformal extension to $\D$ whose
complex dilatation belongs to $\Bel^{\alpha}_0(\D)$. We will show here that
the Douady--Earle extension actually gives such an extension
provided that Theorem \ref{main} is known.  

\begin{theorem}\label{alphasection} 
For every $g \in \Diff_+^{1+\alpha}(\S1)$, 
the image $s_{\rm DE}([g])$ under the conformally natural section
belongs to $\Bel_0^\alpha(\D)$.
\end{theorem}

Let $\sigma: \Bel(\D) \to \Bel(\D)$ be defined by the correspondence of
$\mu$ to $s_{\rm DE}(\pi(\mu))$ for the conformally natural section $s_{\rm DE}$. 
We call this the {\it conformally natural projection} on $\Bel(\D)$.
A crucial property of this projection is the following,
which was proved by Theorem 1 in Cui \cite{Cui}.

\begin{lemma}\label{cui}
Let $\widetilde \mu=(\sigma(\mu^{-1}))^{-1}$ for any $\mu \in \Bel(\D)$. Then,
$$
|\widetilde \mu(w)|^2 \leq 
C_1(1-|w|^2)^2 \int_{\D} \frac{|\mu(z)|^2}{|1-\bar w z|^4} dxdy
$$
for every $w \in \D$, where $C_1=C_1(k)>0$ is a constant depending only on $k$ with $\Vert \mu \Vert_\infty \leq k$.
\end{lemma}

We also need the following claim, which can be found in Lemma 3.10 of Zhu \cite{Z}.

\begin{lemma}\label{GammaBeta}
If $\mu \in \Bel^\alpha_0(\D)$ $(\alpha \in (0,1))$, then
$$
\int_{\D} \frac{|\mu(z)|^2}{|1-\bar w z|^4} dxdy
\leq C_2(1-|w|^2)^{2\alpha-2}
$$
for every $w \in \D$, where $C_2=C_2(\widetilde k)>0$ is a constant depending only on $\widetilde k$ 
with $\Vert \mu \Vert_{\infty,\alpha} \leq \widetilde k$.
\end{lemma}

\medskip
\noindent
{\it Proof of Theorem \ref{alphasection}.}
For $g \in \Diff_+^{1+\alpha}(\S1)$, we choose $\nu \in \Bel^{\alpha}_0(\D)$ such that $\pi(\nu)=[g]$ by Theorem \ref{main}.
Then, $\nu^{-1}$ also belongs to $\Bel^{\alpha}_0(\D)$ by Corollary \ref{inverse}.
For $\mu=\nu^{-1}$, we apply Lemmata \ref{cui} and \ref{GammaBeta} to show that 
$\tilde \mu=(\sigma(\mu^{-1}))^{-1}$ belongs to $\Bel^{\alpha}_0(\D)$. Again by Corollary \ref{inverse},
this shows that $\sigma(\nu) =\sigma(\mu^{-1}) \in \Bel^{\alpha}_0(\D)$. As
$\sigma(\nu)=s_{\rm DE}(\pi(\nu))=s_{\rm DE}([g])$, we have the assertion.
\qed
\medskip

We can also show that the restriction of the conformally natural projection 
$\sigma$ to $\Bel^{\alpha}_0(\D)$ is continuous with respect to the topology
induced by the norm $\Vert \cdot \Vert_{\infty,\alpha}$. 
The detailed proof has been given in \cite{Mat4}.
To see this, we use the relation between the norm $\Vert \cdot \Vert_{\infty,\alpha}$ and
the right uniform topology on $\Diff_+^{1+\alpha}(\S1)$, which will be shown in Theorem \ref{topology}
in the next section.

\section{The Teich\-m\"ul\-ler space of circle diffeomorphisms}\label{8}

We are ready to realize the Teich\-m\"ul\-ler space of circle diffeomorphisms with H\"older continuous derivatives
as a subspace of the universal Teich\-m\"ul\-ler space.
Then, we will give an application of the structure of this space at the end of this section.

\begin{definition}
For a constant $\alpha$ with $0<\alpha <1$,
the Teich\-m\"ul\-ler space of circle diffeomorphisms with $\alpha$-H\"older continuous derivatives is defined by
$$
T_0^\alpha=\Mob(\S1) \backslash \Diff_+^{1+\alpha}(\S1).
$$
\end{definition}

Theorem \ref{main} implies that
the Teich\-m\"ul\-ler projection $\pi:\Bel(\D) \to T$ gives
$$
\pi(\Bel_0^\alpha(\D))=T_0^\alpha,
$$
and the Bers embedding $\beta:T \to B(\D^*)$ gives
$$
\beta(T_0^\alpha)=\beta(T) \cap B_0^\alpha(\D^*),
$$
which coincides with $\Phi(\Bel_0^\alpha(\D))$ for the Bers projection $\Phi:\Bel(\D) \to B(\D^*)$.
Here, we see that $\beta(T) \cap B_0^\alpha(\D^*)$ is an open subset of the Banach space $B_0^\alpha(\D^*)$.
Indeed, this follows from the fact that $\beta(T)$ is open in $B(\D^*)$, and 
the norm inequality $\Vert \varphi \Vert_\infty \leq \Vert \varphi \Vert_{\infty,\alpha}$ for $\varphi \in B_0^\alpha(\D^*)$.

We restrict $\pi$, $\Phi$, and $\beta$ to the spaces as above and consider the
continuity and openness of these maps. We provide $T_0^\alpha$ with the quotient topology from $\Bel_0^\alpha(\D)$ by $\pi$,
which is so defined that $\pi$ is continuous. Then, from the facts listed in the proof below, 
we are able to prove the following:

\begin{theorem}\label{complexst}
The Bers embedding $\beta:T_0^\alpha \to B_0^\alpha(\D^*)$ is a homeomorphism
onto the image $\beta(T) \cap B_0^\alpha(\D^*)$.
Hence, $T_0^\alpha$ is equipped with
the complex structure modeled on the complex Banach space $B_0^\alpha(\D^*)$.
\end{theorem}

\begin{proof}
For the proof of this theorem, it suffices to show the following claims:
\begin{enumerate}
\item
$\Phi:\Bel_0^\alpha(\D) \to B_0^\alpha(\D^*)$ is continuous;
\item
$\Phi:\Bel_0^\alpha(\D) \to \Phi(\Bel_0^\alpha(\D))=\beta(T) \cap B_0^\alpha(\D^*)$ has a local continuous section.
\end{enumerate}
These claims are proved in Lemma \ref{basic} and Lemma \ref{localsection} below,
respectively.
\end{proof}

We begin by showing a basic fact of the group $\Bel_0^\alpha(\D)$.
This is analogous to Proposition 5.1 in Yanagishita \cite{Yan}.

\begin{proposition}\label{r-translation}
The right translation $r_\nu:\Bel_0^\alpha(\D) \to \Bel_0^\alpha(\D)$ for any $\nu \in \Bel_0^\alpha(\D)$
defined by $\mu \mapsto \mu \ast \nu^{-1}$ is a homeomorphism 
with respect to $\Vert \cdot \Vert_{\infty,\alpha}$.
\end{proposition}

\begin{proof}
We have the following formula for $\zeta=f^\nu(z)$:
\begin{align*}
|r_\nu(\mu_1)(\zeta)-r_\nu(\mu_2)(\zeta)|
&= \left|\frac{\mu_1(z)-\nu(z)}{1-\overline{\nu(z)}\mu_1(z)}-\frac{\mu_2(z)-\nu(z)}{1-\overline{\nu(z)}\mu_2(z)}\right|\\
&= \frac{|\mu_1(z)-\mu_2(z)|(1-|\nu(z)|^2)}{|1-\overline{\nu(z)}\mu_1(z)||1-\overline{\nu(z)}\mu_2(z)|}.
\end{align*}
Here, the last term is bounded by
\begin{align*}
&\quad \ \frac{|\mu_1(z)-\mu_2(z)|(1-|\nu(z)|^2)}{\sqrt{(|1-\overline{\nu(z)}\mu_1(z)|^2-|\mu_1(z)-\nu(z)|^2)(|1-\overline{\nu(z)}\mu_2(z)|^2-|\mu_2(z)-\nu(z)|^2)}}\\
&= \frac{|\mu_1(z)-\mu_2(z)|}{\sqrt{(1-|\mu_1(z)|^2)(1-|\mu_2(z)|^2)}}. 
\end{align*} 
By applying Theorem \ref{mori} to $f^\nu$, we have 
$\rho_{\D}^{\alpha}(\zeta) \leq (2A)^\alpha \rho_{\D}^{\alpha}(z)$ for some $A \geq 1$. Hence,
$$
\Vert r_\nu(\mu_1)-r_\nu(\mu_2) \Vert_{\infty,\alpha} \leq \frac{(2A)^\alpha}
{\sqrt{(1-\Vert \mu_1 \Vert_\infty^2)(1-\Vert \mu_2 \Vert_\infty^2)}}
\Vert \mu_1-\mu_2 \Vert_{\infty,\alpha}.
$$
This shows that $r_\nu$ is continuous. As $(r_\nu)^{-1}=r_{\nu^{-1}}$ and $\nu^{-1} \in \Bel_0^\alpha(\D)$, 
we also see that the inverse $(r_\nu)^{-1}$ is continuous.
\end{proof}

We note that the right translation $r_\nu$ for $\nu \in \Bel_0^\alpha(\D)$ projects down to the base point change map
$R_{\pi(\nu)}:T^\alpha_0 \to T^\alpha_0$; then,  
as $r_\nu$ is a homeomorphism, so is $R_{\pi(\nu)}$.
Their holomorphy will be discussed later in Corollary \ref{basepointchange}.

The continuity of $\Phi:\Bel_0^\alpha(\D) \to B_0^\alpha(\D^*)$ can be proved as a special case of
the assertion that follows. In contrast to the original case (Proposition \ref{B-conti}),
we need to introduce here a certain representation of a Schwarzian derivative using Beltrami coefficients
and estimate it by the results which we have obtained.

\begin{lemma}\label{basic}
Let $\nu \in \Bel^{\alpha'}_0(\D)$ possibly with $\alpha \neq \alpha'\in (0,1)$.
Then, every $\mu \in \Bel(\D)$ satisfies
$$
\Vert \Phi(\mu) -\Phi(\nu)\Vert_{\infty, \alpha}
\leq C \Vert \mu-\nu \Vert_{\infty,\alpha},
$$
where $C=C(\nu,\alpha,k)>0$ is a constant depending only on 
$\nu$, $\alpha$, and $k$ with $\Vert \mu \Vert_\infty \leq k$.
The dependence on $\nu$ is further given by $\alpha'$, $\Vert \nu \Vert_\infty$, and $\Vert \nu \Vert_{\infty,\alpha'}$.
The term on the right side is assumed to be $\infty$ when $\mu-\nu \notin \Bel^\alpha_0(\D)$.
\end{lemma}

The following integral representation of Schwarzian derivatives can be found in 
Lemma 3.1 and Proposition 3.2 of Yanagishita \cite{Yan},
which is obtained by generalizing the arguments in Astala and Zinsmeister \cite{AZ}.

\begin{proposition}\label{yanagishita}
For Beltrami coefficients $\mu$ and $\nu$ in $\Bel(\D)$, let $f_\mu$ and $f_\nu$ be the normalized quasiconformal self-homeomorphisms of $\Chat$
that are conformal on $\D^*$. Let $\Omega=f_\nu(\D)$ and $\Omega^*=f_\nu(\D^*)$. Then,
\begin{align*}
|S_{f_\mu \circ f_\nu^{-1}|_{\Omega^*}}(\zeta)| 
&\leq \frac{3\rho_{\Omega^*}(\zeta)}{\sqrt{\pi}}
\left(\int_\Omega \frac{|\mu(f_\nu^{-1}(w))-\nu(f_\nu^{-1}(w))|^2}{(1-|\mu(f_\nu^{-1}(w))|^2)(1-|\nu(f_\nu^{-1}(w))|^2)}
\frac{dudv}{|w-\zeta|^4}\right)^{1/2}
\end{align*}
holds for every $\zeta \in \Omega^*$.
\end{proposition}

To consider the norm of the Schwarzian derivative $\Phi(\mu)=S_{f_\mu|_{\D^*}}$, we need
an estimate of the derivative of the conformal homeomorphism $f_\mu$ of $\D^*$ defined by $\mu \in \Bel_0^\alpha(\D)$.
We use Proposition \ref{distortion} for this purpose.

\medskip
\noindent
{\it Proof of Lemma \ref{basic}.}
By the definition of the norm,
$$
|\mu(z)-\nu(z)| \leq  \rho_{\D}^{-\alpha}(z) \Vert \mu-\nu \Vert_{\infty,\alpha}
$$
for every $z \in \D$. 
By Theorem \ref{mori}, 
there is a constant $a=a(\alpha,\nu) \geq 1$ 
such that
$\rho_{\D}^{-\alpha}(z) \leq a \rho_{\D}^{-\alpha}(f^\nu(z))$.

Let
$f=f_\nu \circ (f^\nu)^{-1}$. This is a conformal homeomorphism of $\D$ extending to a
quasiconformal homeomorphism of $\Chat$ whose complex dilatation on $\D^*$ coincides with $(\nu^*)^{-1}$. 
We can choose $f_\nu$ so that $f(0)=0$ maintaining the normalization
$f_{\nu}(\infty)=\infty$ and $\lim_{z \to \infty} f'_\nu(z)=1$.
We note that $f(\D)=f_\nu(\D)$. 

If the normalization of $f_\nu$ appeals to the Schwarz lemma and the
Koebe one-quarter theorem (Proposition \ref{Koebe}) on $\D^*$, we see that
$f_\nu(\D)$ is not strictly contained in $\D$ but is instead contained in the disk $\{|z|<4\}$. Hence, 
there is some $x_1 \in \S1$ such that
$1 \leq |f_\nu(x_1)| \leq 4$. Furthermore, Proposition \ref{Teich} asserts that
there is some $r \in [0,1)$ 
depending only on $\Vert \nu^{-1} \Vert_\infty=\Vert \nu \Vert_\infty$ such that $|(f^{\nu})^{-1}(0)| \leq r$.
We take $z \in \D$ with $|z|=(1+r)/2$ arbitrarily and consider the cross ratio
$[(f^{\nu})^{-1}(0),x_1,\infty,z]$. 
By the distortion theorem for cross ratio due to Teich\-m\"ul\-ler (see Section III.D of \cite{Ah0} and \cite{Kra}), 
the hyperbolic distance on $\C-\{0,1\}$ between $[(f^{\nu})^{-1}(0),x_1,\infty,z]$ and 
$$
[f_\nu((f^{\nu})^{-1}(0)),f_\nu(x_1),f_\nu(\infty),f_\nu(z)]=[0,f_\nu(x_1),\infty,f_\nu(z)]
$$ 
is bounded by $\log K$, 
where $K=(1+\Vert \nu \Vert_{\infty})/(1-\Vert \nu \Vert_{\infty})$.
This implies that there is a constant $\rho=\rho(\Vert \nu \Vert_{\infty})>0$ 
such that
$|f_\nu(z)| \geq \rho$ for $|z|=(1+r)/2$, and hence $f(\D)=f_\nu(\D)$ contains the disk of center at $0$ and 
radius $\rho$. 

By the Schwarz lemma applied to the conformal homeomorphism $f$ of $\D$,
we see that there is a constant $s=s(\rho)>0$ depending only on $\rho$ and hence on $\Vert \nu \Vert_{\infty}$ such that
$e^{-s} \leq |f'(0)| \leq 4$. 
It follows from
Proposition \ref{distortion} that there is a constant $B=B(\nu)>0$ such that
$|f'(z)| \geq 1/B$ for every $z \in \D$.
Hence,
there is a constant $b=b(\nu,\alpha) \geq 1$ such that
$\rho_{\D}^{-\alpha}(f^\nu(z)) \leq b \rho_{\Omega}^{-\alpha}(f_\nu(z))$.

For $w=f_\nu(z) \in \Omega$, this inequality and $\rho_{\D}^{-\alpha}(z) \leq a \rho_{\D}^{-\alpha}(f^\nu(z))$ yield that
$$
|\mu(f_\nu^{-1}(w))-\nu(f_\nu^{-1}(w))| \leq ab \rho_{\Omega}^{-\alpha}(w) \Vert \mu-\nu \Vert_{\infty,\alpha}.
$$
By substituting this inequality into the integral in Proposition \ref{yanagishita}, we will estimate 
$$
\left(\int_\Omega \frac{\rho_{\Omega}^{-2\alpha}(w)}{|w-\zeta|^4}dudv\right)^{1/2}.
$$
We follow an estimation procedure similar to that in Section 3.4 of Nag \cite{Nag}.
Let $\eta_\Omega(w)$ be the Euclidean distance from $w \in \Omega$ to $\partial \Omega$ and
$\eta_{\Omega^*}(\zeta)$ the Euclidean distance from $\zeta \in \Omega^*$ to $\partial \Omega$.
We see, as a consequence of the Koebe one-quarter theorem (Proposition \ref{Koebe}), that both $\rho_{\Omega}(w)\eta_{\Omega}(w)$
and $\rho_{\Omega^*}(\zeta)\eta_{\Omega^*}(\zeta)$ are bounded below by $1/2$.
We have
$$
\rho_{\Omega}^{-2\alpha}(w) \leq 4\eta_{\Omega}^{2\alpha}(w) \leq 4|w-\zeta|^{2\alpha}
$$
for every $w \in \Omega$ and every $\zeta \in \Omega^*$.
Hence, the integral can be estimated as
\begin{align*}
\int_\Omega \frac{\rho_{\Omega}^{-2\alpha}(w)}{|w-\zeta|^4}dudv &\leq  4\int_\Omega \frac{dudv}{|w-\zeta|^{4-2\alpha}}\\
&\leq
4\int_{|w-\zeta| \geq \eta_{\Omega^*}(\zeta)}
\frac{dudv}{|w-\zeta|^{4-2\alpha}}\\
&=
\frac{8\pi}{2-2\alpha} \cdot
\frac{1}{\eta_{\Omega^*}(\zeta)^{2-2\alpha}}
\leq
\frac{16\pi}{1-\alpha}\rho_{\Omega^*}(\zeta)^{2-2\alpha}.
\end{align*}

Plugging this estimate in the inequality of Proposition \ref{yanagishita}, we have
$$
\rho^{-2}_{\Omega^*}(\zeta)|S_{f_\mu \circ f_\nu^{-1}|_{\Omega^*}}(\zeta)| \leq 
\frac{12ab\Vert \mu-\nu \Vert_{\infty,\alpha}}{\sqrt{(1-\alpha)(1-\Vert \mu \Vert_\infty^2)(1-\Vert \nu \Vert_\infty^2)}}\rho^{-\alpha}_{\Omega^*}(\zeta).
$$
For $\zeta=f_\nu(z)$ with $z \in \D^*$, the left side term is equal to
$$
\rho^{-2}_{\D^*}(z)|S_{f_\mu|_{\D^*}}(z)-S_{f_\nu|_{\D^*}}(z)|.
$$
For the right side term, we apply Proposition \ref{distortion} again to
the quasiconformal homeomorphism $f_\nu$ of $\Chat$ that is conformal on $\D^*$. 
Then, there is a constant $b'=b'(\nu,\alpha) \geq 1$ 
such that
$\rho^{-\alpha}_{\Omega^*}(f_\nu(z)) \leq b' \rho^{-\alpha}_{\D^*}(z)$. Therefore, the above inequality turns out to be
$$
\rho^{-2}_{\D^*}(z)|S_{f_\mu|_{\D^*}}(z)-S_{f_\nu|_{\D^*}}(z)| \leq 
\frac{12abb'\Vert \mu-\nu \Vert_{\infty,\alpha}}{\sqrt{(1-\alpha)(1-\Vert \mu \Vert_\infty^2)(1-\Vert \nu \Vert_\infty^2)}}\rho^{-\alpha}_{\D^*}(z).
$$
This implies that
$$
\Vert \Phi(\mu)-\Phi(\nu) \Vert_{\infty,\alpha} \leq
\frac{12abb'}{\sqrt{(1-\alpha)(1-\Vert \mu \Vert_\infty^2)(1-\Vert \nu \Vert_\infty^2)}} \Vert \mu-\nu \Vert_{\infty,\alpha}.
$$
We can choose the multiplier of the term on the right side as the constant $C$. 
\qed
\medskip

The existence of a local continuous section for $\Phi:\Bel_0^\alpha(\D) \to \beta(T) \cap B_0^\alpha(\D^*)$
is verified similarly to the original Bers projection $\Phi$, for which
the local holomorphic section was defined by using the quasiconformal reflection proposed by Ahlfors \cite{Ah}. 
This was improved later by Earle and Nag \cite{EN}.

\begin{lemma}\label{localsection}
The Bers projection
$\Phi:\Bel_0^\alpha(\D) \to \beta(T) \cap B_0^\alpha(\D^*)$ has a local continuous section
at every $\psi \in \beta(T) \cap B_0^\alpha(\D^*)$.
\end{lemma}

\begin{proof}
By Theorem \ref{alphasection}, we can take $\nu \in \Bel_0^\alpha(\D)$ 
in the image of the conformally natural projection such that $\Phi(\nu)=S_{f_{\nu}|_{\D^*}}=\psi$ and
$f_{\nu}|_{\D}$ is a bi-Lipschitz diffeomorphism
with respect to the
hyperbolic metric 
(Theorem 2 of \cite{DE}). 
The quasiconformal reflection
$\lambda:f_{\nu}(\D) \to f_{\nu}(\D^*)$ with respect to the quasicircle $f_{\nu}(\S1)$ is defined by
$\lambda(\zeta)=f_{\nu}(f_{\nu}^{-1}(\zeta)^*)$,
where $z^*$ denotes the reflection of $z$ with respect to $\S1$. 

We follow the arguments in Section II.4.2 of \cite{Leh} and Section 14.3--4 of \cite{GL}. 
We have a constant $\varepsilon=\varepsilon(k)>0$ depending only on $k$ with $\Vert \nu \Vert_{\infty} \leq k$
such that if $\varphi \in B(\D^*)$ satisfies $\Vert \varphi \Vert_{\infty} <\varepsilon$, then
there is a quasiconformal self-homeomorphism $\widehat f$ of $\Chat$ conformal on $f_{\nu}(\D^*)$ such that
$S_{\widehat f \circ f_{\nu}|_{\D^*}}=\psi+\varphi$ (see also Theorem III.4.2 in \cite{Leh}).
In this case, 
the Beltrami coefficient $\mu_{\widehat f}$ of $\widehat f$ 
is given by
\begin{align*}
\mu_{\widehat f}(\zeta)
&=\frac{S_{\widehat f}(\lambda(\zeta))(\zeta-\lambda(\zeta))^2 \bar \partial \lambda(\zeta)}
{2+S_{\widehat f}(\lambda(\zeta))(\zeta-\lambda(\zeta))^2 \partial \lambda(\zeta)}
\end{align*}
for $\zeta \in f_{\nu}(\D)$. Here, by the bi-Lipschitz property, we see that
$$
|(\zeta-\lambda(\zeta))^2 \partial \lambda(\zeta)| \leq 
|(\zeta-\lambda(\zeta))^2 \bar \partial \lambda(\zeta)| \leq c \rho_{f_{\nu}(\D^*)}^{-2}(\lambda(\zeta))
$$
for some constant $c=c(k)>0$. Then, by replacing $\varepsilon>0$ so that
$\varepsilon \leq 1/c$ if necessary, we have 
\begin{align*}
&\quad\ |S_{\widehat f}(\lambda(\zeta))(\zeta-\lambda(\zeta))^2 \partial \lambda(\zeta)|\\
&=|\varphi(f_{\nu}^{-1}(\lambda(\zeta)))((f_{\nu}^{-1})'(\lambda(\zeta)))^2(\zeta-\lambda(\zeta))^2 \partial \lambda(\zeta)|\\
&\leq c|\varphi(z^*)||f_{\nu}'(z^*)|^{-2} \rho_{f_{\nu}(\D^*)}^{-2}(f_{\nu}(z^*))
\leq \frac{1}{\varepsilon}\,|\varphi(z^*)| \rho_{\D^*}^{-2}(z^*) <1
\end{align*}
for $\zeta =f_{\nu}(z)$; hence
\begin{align*}
|\mu_{\widehat f}(f_{\nu}(z))| 
&\leq |S_{\widehat f}(\lambda(\zeta))(\zeta-\lambda(\zeta))^2 \bar \partial \lambda(\zeta)|
\leq \frac{1}{\varepsilon}\,|\varphi(z^*)| \rho_{\D^*}^{-2}(z^*)
\end{align*}
for every $z \in \D$. 

Now we take $\varphi \in B_0^\alpha(\D^*)$ such that 
$\Vert \varphi \Vert_{\infty} \leq \Vert \varphi \Vert_{\infty,\alpha} <\varepsilon$.
Then, we can apply the above argument to obtain 
$$
|\mu_{\widehat f}(f_{\nu}(z))| \leq \frac{1}{\varepsilon}\,|\varphi(z^*)| \rho_{\D^*}^{-2}(z^*)
\leq \frac{\Vert \varphi \Vert_{\infty,\alpha}}{\varepsilon}\, \rho_{\D^*}^{-\alpha}(z^*)
< |z|^{-2\alpha} \rho_{\D}^{-\alpha}(z).
$$
We use this estimate when $|z|$ is bounded away from $0$. 
When $|z|$ is small, for example, 
if $|z| < 1/\sqrt{2}$, then
$|\mu_{\widehat f}(f_{\nu}(z))| <1 < 4\rho_{\D}^{-\alpha}(z)$.
Thus, we see that $\mu_{\widehat f} \circ f_{\nu} \in \Bel_0^\alpha(\D)$.

Denoting the complex dilatation of $\widehat f \circ f_{\nu}$ by $\mu_{\varphi}$,  we will
show that $\mu_{\varphi} \in \Bel_0^\alpha(\D)$.
The formula for the complex dilatation of composed quasiconformal homeo\-morphisms is 
$$
\mu_{\varphi}(z)=\frac{e^{-2i\theta}\mu_{\widehat f}(f_\nu(z))+\nu(z)}{1+e^{-2i\theta}\mu_{\widehat f}(f_\nu(z))\overline{\nu(z)}}
\quad (z \in \D),
$$
where $\theta={\rm arg}\,\partial f_\nu(z)$. Then, similarly to the proof of Proposition \ref{r-translation}, we have
$$
|\mu_{\varphi}(z)-\mu_{\varphi'}(z)| \leq \frac{|\mu_{\widehat f} \circ f_{\nu}(z)-\mu_{\widehat f'} \circ f_{\nu}(z)|}{\sqrt{(1-\Vert\mu_{\widehat f} \Vert_\infty^2)(1-\Vert \mu_{\widehat f'}\Vert_\infty^2)}}
$$
for any $\varphi$ and $\varphi'$ in $B_0^\alpha(\D^*)$ with $\Vert \varphi \Vert_{\infty,\alpha},
\Vert \varphi' \Vert_{\infty,\alpha} <\varepsilon$, where $\widehat f$ and $\widehat f'$ are the corresponding
quasiconformal homeomorphisms. In particular, setting $\varphi'=0$ yields
$$
|\mu_{\varphi}(z)-\nu(z)| \leq \frac{|\mu_{\widehat f} \circ f_{\nu}(z)|}{\sqrt{(1-\Vert \mu_{\widehat f}\Vert_\infty^2)}}.
$$
As both $\mu_{\widehat f} \circ f_{\nu}$ and $\nu$ belong to $\Bel_0^\alpha(\D)$, so does $\mu_{\varphi}$.

Because
$\Phi(\mu_\varphi)=S_{\widehat f \circ f_{\nu}|_{\D^*}}=\psi+\varphi$,  
we have a local section $\eta$ of $\Phi$ on the neighborhood
$$
U(\psi,\varepsilon)=\{\psi+\varphi \mid \Vert \varphi \Vert_{\infty,\alpha}<\varepsilon\} \subset \beta(T) \cap B_0^\alpha(\D^*)
$$
by the correspondence $\eta:\psi+\varphi \mapsto \mu_\varphi$.
By the above inequalities for $\mu_{\varphi}$ and $\mu_{\widehat f} \circ f_{\nu}$, we see that
$$
|\mu_{\varphi}(z)-\mu_{\varphi'}(z)| \leq C |\varphi(z^*)-\varphi'(z^*)| \rho_{\D^*}^{-2}(z^*) \quad (z \in \D)
$$
for some constant $C>0$. This implies that $\eta$ is continuous.
\end{proof}

We have obtained the continuity of the Bers projection $\Phi$ and its local section
restricted to $\Bel_0^\alpha(\D)$ and $\beta(T) \cap B_0^\alpha(\D^*)$, respectively,
with respect to the norm $\Vert \cdot \Vert_{\infty,\alpha}$.
We note that the local section at $\psi \in \beta(T) \cap B_0^\alpha(\D^*)$ can be chosen so that $\psi$ is sent to an arbitrary point
in the fiber $\Phi^{-1}(\psi)$ by post-composition of the right translation $r_\lambda$ for $\lambda \in \pi^{-1}([\id])$.

These maps are given in the same form as the original ones for $\Bel(\D)$ and $B(\D^*)$.
Moreover, we know that these maps are holomorphic on $\Bel(\D)$ and $B(\D^*)$ with respect to the norm $\Vert \cdot \Vert_{\infty}$
(Theorem \ref{hol-sub}).
Once we are in this situation, to see that the new maps are in fact holomorphic is a matter of general argument.
Indeed, $\Phi$ and its local section are holomorphic as mappings to $\C$ 
if we fix the complex variable $z$ of functions $\varphi(z) \in B_0^\alpha(\D^*)$ and $\mu(z) \in \Bel_0^\alpha(\D)$.
Then, the norm inequality $\Vert \cdot \Vert_{\infty} \leq \Vert \cdot \Vert_{\infty,\alpha}$
and the continuity under $\Vert \cdot \Vert_{\infty,\alpha}$ justify the claim (see Lemma 3.4 in Earle \cite{E1} 
and Lemma V.5.1 in Lehto \cite{Leh}).

\begin{theorem}
The Bers projection $\Phi:\Bel_0^\alpha(\D) \to \beta(T) \cap B_0^\alpha(\D^*)$ is a holomorphic split submersion. 
\end{theorem}

Moreover, we have seen in Proposition \ref{r-translation} that the right translation and, hence, the base point change map 
are homeomorphisms.
By the same reasoning as above, we also see that they are biholomorphic.

\begin{corollary}\label{basepointchange}
The right translation $r_\nu:\Bel_0^\alpha(\D) \to \Bel_0^\alpha(\D)$ for $\nu \in \Bel_0^\alpha(\D)$ and
the base point change map $R_\tau: T^\alpha_0 \to T^\alpha_0$ for $\tau \in T^\alpha_0$ are biholomorphic.
\end{corollary}

The Teich\-m\"ul\-ler space $T^\alpha_0$ is equipped with the Kobayashi metric as an invariant metric of
a complex manifold. By Theorem 1.1 in Yanagishita \cite{Yan0}, which generalizes the result
of Hu, Jiang, and Wang \cite{HJW}, we see that
the Kobayashi distance on $T^\alpha_0$ coincides with the restriction of the Teich\-m\"ul\-ler distance on $T$ to $T^\alpha_0$;
hence, the infinitesimal Kobayashi metric on each tangent space of $T^\alpha_0$ coincides with
its restriction of the infinitesimal Teich\-m\"ul\-ler metric on the tangent space of $T$.

Finally, in this section, we investigate the topology on $T^\alpha_0$, which has been defined 
as the quotient topology 
induced from the norm $\Vert \cdot \Vert_{\infty,\alpha}$ on $\Bel_0^\alpha(\D)$ by the Teich\-m\"ul\-ler projection $\pi$.
However, as $\Diff_+^{1+\alpha}(\S1)$ is equipped
with the right uniform topology, we can also introduce another topology on $T^\alpha_0$.
This topology is 
the relative topology under the 
identification of $T^\alpha_0$ with the subgroup $\Diff_+^{1+\alpha}(\S1) \cap \QS_*$ of all normalized elements.
We also call this the right uniform topology on $T^\alpha_0$.
Concerning the relation between these two topologies on $T^\alpha_0$, we have the following:

\begin{theorem}\label{topology}
The right uniform topology on $T^\alpha_0$ coincides with the quotient topology induced from $\Bel_0^\alpha(\D)$.
\end{theorem}

\begin{proof}
Suppose that $[g_n] \to [g]$ as $n \to \infty$ in the quotient topology on $T^\alpha_0$
for $g_n, g \in \Diff_+^{1+\alpha}(\S1) \cap \QS_*$. Then, there are $\mu_n$ and $\mu$ in $\Bel^\alpha_0(\D)$ with $\pi(\mu_n)=[g_n]$ and
$\pi(\mu)=[g]$ such that $\mu_n \to \mu$ with respect to $\Vert \cdot \Vert_{\infty,\alpha}$.
As the right translation $r_\mu$ is a homeomorphism of
$\Bel^\alpha_0(\D)$ by Proposition \ref{r-translation}, 
the condition $\mu_n \to \mu$ is equivalent to the condition $r_\mu(\mu_n)=\mu_n \ast \mu^{-1} \to 0$ in $\Bel^\alpha_0(\D)$.
Then, by Theorem \ref{key}, 
the normalized representatives $\gamma_n=g_n \circ g^{-1} \in \Diff_+^{1+\alpha}(\S1)\cap \QS_*$ with $[\gamma_n]=\pi(r_\mu(\mu_n))$
satisfy $p_{1+\alpha}(\gamma_n) \to 0$ as $n \to \infty$.
This means that
$\gamma_n$ converge to $\id$ in $\Diff_+^{1+\alpha}(\S1) \cap \QS_*$.
Hence, $[g_n]$ converge to $[g]$ in the right uniform topology on $T^\alpha_0$. 

Conversely, suppose that $[g_n] \to [g]$ as $n \to \infty$ in the right uniform topology on $T^\alpha_0$ for $g_n, g \in \Diff_+^{1+\alpha}(\S1) \cap \QS_*$. 
Then, $\gamma_n=g_n \circ g^{-1}$ converge to $\id$, 
that is,
$p_{1+\alpha}(\gamma_n) \to 0$. 
In particular, $c_\alpha(\gamma_n) \to 0$.
Then, by Theorem \ref{diff-qc},
we have quasiconformal extensions $f_n:\D \to \D$ of $\gamma_n$, whose complex dilatations $\nu_n$ satisfy
$\Vert \nu_n \Vert_{\infty,\alpha} \to 0$ as $n \to \infty$. Hence, $[\gamma_n]=[g_n] \ast [g]^{-1} \to [\id]$ in 
the quotient topology on $T^\alpha_0$, and thus
$[g_n] \to [g]$ by the continuity of 
the base point change map $R_{[g]^{-1}}$ of $T^\alpha_0$.
\end{proof}

Combined with Proposition \ref{top}, this implies the following:

\begin{corollary}
$(T^\alpha_0,\ast)$ is a topological group.
\end{corollary}

\end{document}